\documentclass[a4paper, 10pt, american]{amsart}

\usepackage[colorlinks,pdfpagelabels,pdfstartview = FitH,bookmarksopen
= true,bookmarksnumbered = true,linkcolor = blue,plainpages =
false,hypertexnames = false,citecolor = red,pagebackref=false]{hyperref}
\usepackage{times,latexsym,amssymb}
\usepackage{amsmath,amsthm,bm}
\usepackage{graphics, epsfig}
\usepackage{color}
\usepackage{appendix}
\usepackage[margin=1.5in]{geometry}

\usepackage{amsbsy}
\usepackage{amstext}
\usepackage{amssymb}
\usepackage{esint}
\usepackage{stmaryrd}
\allowdisplaybreaks

\let\TeXchi\chi
\newbox\chibox
\setbox0 \hbox{\mathsurround0pt $\TeXchi$}
\setbox\chibox \hbox{\raise\dp0 \box 0 }
\def\chi{\copy\chibox}

\renewcommand{\d}{\mathrm{d}}
\newcommand{\dx}{\mathrm{d}x}

\newcommand{\dt}{\mathrm{d}t}
\newcommand{\ds}{\mathrm{d}s}

\renewcommand{\epsilon}{\varepsilon}
\renewcommand{\rho}{\varrho}

\DeclareFontFamily{U}{mathx}{\hyphenchar\font45}
\DeclareFontShape{U}{mathx}{m}{n}{
      <5> <6> <7> <8> <9> <10>
      <10.95> <12> <14.4> <17.28> <20.74> <24.88>
      mathx10
      }{}
\DeclareSymbolFont{mathx}{U}{mathx}{m}{n}
\DeclareFontSubstitution{U}{mathx}{m}{n}
\DeclareMathAccent{\widecheck}{0}{mathx}{"71}
\DeclareMathAccent{\wideparen}{0}{mathx}{"75}

\author[V. B\"ogelein]{Verena B\"{o}gelein}
\address{Verena B\"ogelein\\
Fachbereich Mathematik, Universit\"at Salzburg\\
Hellbrunner Str. 34, 5020 Salzburg, Austria}
\email{verena.boegelein@sbg.ac.at}

\author[F. Duzaar]{Frank Duzaar}
\address{Frank Duzaar\\
Department Mathematik, Universit\"at Erlangen--N\"urnberg\\
Cauerstrasse 11, 91058 Erlangen, Germany}
\email{frank.duzaar@fau.de}

\author[N. Liao]{Naian Liao}
\address{Naian Liao\\
Fachbereich Mathematik, Universit\"at Salzburg\\
Hellbrunner Str. 34, 5020 Salz\-burg, Austria}
\email{naian.liao@sbg.ac.at}

\author[L. Sch\"atzler]{Leah Sch\"atzler}
\address{Leah Sch\"atzler\\
Fachbereich Mathematik, Universit\"at Salzburg\\
Hellbrunner Str. 34, 5020 Salzburg, Austria}
\email{leahanna.schaetzler@sbg.ac.at}



\begin{document}
\newtheorem{proposition}{Proposition}[section]
\newtheorem{theorem}{Theorem}[section]
\newtheorem{lemma}{Lemma}[section]
\newtheorem{corollary}{Corollary}[section]
\newtheorem{remark}{Remark}[section]
\renewcommand{\thesection}{\arabic{section}}
\renewcommand{\theequation}{\thesection.\arabic{equation}}
\renewcommand{\thetheorem}{\thesection.\arabic{theorem}}
\numberwithin{equation}{section}
\numberwithin{theorem}{section}
\numberwithin{proposition}{section}
\numberwithin{lemma}{section}
\numberwithin{remark}{section}
\setcounter{secnumdepth}{3}
\newcommand{\cl}{\centerline}
\newcommand{\sms}{\smallskip}
\newcommand{\ms}{\medskip}
\newcommand{\bs}{\bigskip}
\newcommand{\noi}{\noindent}
\newcommand{\itl}[1]{\textit{#1}}
\newcommand{\blf}[1]{\textbf{#1}}
\newcommand{\dsty}{\displaystyle}
\newcommand{\txty}{\textstyle}
\newcommand{\ssty}{\scriptstyle}
\newcommand{\tty}{\texttt}


\newcommand\Par{\mathhexbox278\,}


\newcommand{\al}{\alpha}
\newcommand{\Al}{\Alpha}
\newcommand{\be}{\beta}
\newcommand{\Be}{\Beta}
\newcommand{\Gm}{\Gamma}
\newcommand{\gm}{\gamma}
\newcommand{\dl}{\delta}
\newcommand{\Dl}{\Delta}
\newcommand{\lm}{\lambda}
\newcommand{\Lm}{\Lambda}
\newcommand{\kp}{\kappa}
\newcommand{\varep}{\varepsilon}
\newcommand{\eps}{\epsilon}
\newcommand{\vp}{\varphi}
\newcommand{\sig}{\sigma}
\newcommand{\Sig}{\Sigma}
\newcommand{\om}{\omega}
\newcommand{\Om}{\Omega}
\newcommand{\uom}{\mbox{\boldmath$\omega$}}
\newcommand{\btau}{\mbox{\boldmath$\tau$}}
\newcommand{\bnu}{\mbox{\boldmath$\nu$}}
\newcommand{\up}{\upsilon}
\newcommand{\z}{\zeta}


\newcommand{\df}[1]{\buildrel\mbox{\small def}\over{#1}}
\newcommand{\op}[1]{\buildrel\mbox{\tiny o}\over{#1}}
\newcommand{\db}{\prime\prime}
\newcommand{\bsl}{\backslash}
\newcommand{\lb}{\lbrack\!\lbrack}
\newcommand{\rb}{\rbrack\!\rbrack}
\newcommand\la{\langle}
\newcommand\ra{\rangle}
\newcommand{\ev}{\equiv}
\newcommand{\nev}{\not\equiv}
\newcommand{\nn}{\mathbb{N}}
\newcommand{\qq}{\mathbb{Q}}
\newcommand{\zz}{\mathbb{Z}}
\newcommand{\rr}{\mathbb{R}}
\newcommand{\rn}{\rr^N}
\newcommand{\cc}{\mathbb{C}}
\newcommand{\id}{\mathbb{I}}
\newcommand{\bo}{\mathbb{O}}

\newcommand{\amsb}[1]{\mathbb{#1}}
\newcommand{\mcl}[1]{\mathcal{#1}}
\newcommand{\bl}[1]{\mathbf{#1}}
\newcommand{\ov}[1]{\overline{#1}}
\newcommand{\wt}[1]{\widetilde{#1}}
\newcommand{\wh}[1]{\widehat{#1}}

\newcommand{\llra}{\leftrightarrow}
\newcommand{\lra}{\longrightarrow}
\newcommand{\LLR}{\Longleftrightarrow}
\newcommand{\LRA}{\Longrightarrow}
\newcommand{\LLA}{\Longleftarrow}


\newcommand{\bbox}{\vrule height.6em width.6em 
depth0em} 
\newcommand{\os}{\vbox{\hrule \hbox{\vrule 
height.6em depth0pt 
\hskip.6em \vrule height.6em depth0em}
\hrule}} 


\newcommand{\dvg}{\operatorname{div}}
\newcommand{\curl}{\operatorname{curl}}
\newcommand{\supp}{\operatorname{supp}}
\newcommand{\essup}{\operatornamewithlimits{ess\,sup}}
\newcommand{\essinf}{\operatornamewithlimits{ess\,inf}}
\newcommand{\essosc}{\operatornamewithlimits{ess\,osc}}
\newcommand{\osc}{\operatornamewithlimits{osc}}
\newcommand{\sign}{\operatorname{sign}}
\newcommand{\loc}{\operatorname{loc}}
\newcommand{\diam}{\operatorname{diam}}
\newcommand{\dist}{\operatorname{dist}}
\newcommand{\card}{\operatorname{card}}
\newcommand{\meas}{\operatorname{meas}}
\newcommand{\spn}{\operatorname{span}}
\newcommand{\dtm}{\operatorname{det}}
%


\newcommand{\overlim}{\mathop{\overline{\lim}}\limits}
\newcommand{\underlim}{\mathop{\underline{\lim}}\limits}
\newcommand{\ttop}[2]{\genfrac{}{}{0pt}{}{#1}{#2}}
\newcommand{\bcu}{\mathop{\txty{\bigcup}}\limits}
\newcommand{\bca}{\mathop{\txty{\bigcap}}\limits}
\newcommand{\bsu}{\mathop{\txty{\sum}}\limits}
\newcommand{\pro}{\mathop{\txty{\prod}}\limits}


\newcommand{\pl}{\partial}
\newcommand{\ptt}{\frac{\pl}{\pl t}}
\newcommand{\ppx}{\frac\pl{\pl x}}
\newcommand{\dds}{\frac d{ds}}
\newcommand{\ddt}{\frac d{dt}}

\newcommand{\intl}{\int\limits}
\newcommand{\iintl}{\iint\limits}
\def\Xint#1{\mathchoice
    {\XXint\displaystyle\textstyle{#1}}%
    {\XXint\textstyle\scriptstyle{#1}}%
    {\XXint\scriptstyle\scriptscriptstyle{#1}}%
    {\XXint\scriptscriptstyle\scriptscriptstyle{#1}}%
    \!\int}
\def\XXint#1#2#3{\setbox0=\hbox{$#1{#2#3}{\int}$}
    \vcenter{\hbox{$#2#3$}}\kern-0.5\wd0}
\def\bint{\Xint-}
\def\dashint{\Xint{\raise4pt\hbox to7pt{\hrulefill}}}
\def\dashiint{\bint\kern-0.15cm\bint}

\newcommand{\ovl}[3]{\int_{#1}^{#2}\kern-#3pt\raise4pt\hbox to7pt{\hrulefill}\ }

\newcommand{\ovll}[3]{\intl_{#1}^{#2}\kern-#3pt\raise4pt\hbox to7pt{\hrulefill}\ }

\newcommand{\tvl}[2]{\iint_{#1}\kern-#2pt\raise4pt\hbox to7pt{\hrulefill}\ }



\newcommand{\omt}{\Om_T}
\newcommand{\plo}{\partial\Omega}
\newcommand{\ovo}{\bar{\Om} }

%
\newcommand{\ci}[1]{C^\infty\!\left({#1}\right)}
\newcommand{\cio}[1]{C_o^\infty\!\left({#1}\right)}
\newcommand{\lloc}[1]{L_{\loc}\!\left({#1}\right)}
\newcommand{\xy}{|x-y|}


\newcommand{\intom}{\intl_{\Om}}
\newcommand{\intbo}{\intl_{\plo}}
\newcommand{\inom}{\int_{\Om}}
\newcommand{\inbo}{\int_{\plo}}
\newcommand{\intrn}{\intl_{\rn}}


\newcommand{\bye}{\end{document}}



%
%
%

%
%
%
%
%

%
%

\title[H\"older regularity for a doubly nonlinear equation]{On the H\"older regularity of signed solutions to a doubly nonlinear equation. part II}

\date{}
\begin{abstract}
We demonstrate two proofs for the local H\"older continuity of possibly sign-changing solutions to
a class of doubly nonlinear parabolic equations whose prototype is
\[
\partial_t\big(|u|^{q-1}u\big)-\Dl_p u=0,\quad p>2,\quad 0<q<p-1.
\]
The first proof takes advantage of the expansion of positivity for the degenerate, parabolic $p$-Laplacian,
 thus simplifying the argument; whereas the other proof
relies solely on the energy estimates for the doubly nonlinear parabolic equations.
After proper adaptions of the interior arguments, we also obtain the boundary regularity for
 initial-boundary value problems of Dirichlet type and Neumann type.
\vskip.2truecm
\end{abstract}
\maketitle

\tableofcontents

\section{Introduction and Main Results}
Initiated in \cite{BDL}, we continue our study on the H\"older regularity of weak solutions to
a class of doubly nonlinear parabolic equations whose model case is
\begin{equation}\label{prototype}
	\partial_t\big(|u|^{q-1}u\big)- \dvg \big(|Du|^{p-2}Du\big) =0\quad\mbox{ weakly in $ E_T$.}
\end{equation}
Here $E_T:=E\times(0,T]$ for some $T>0$  and some $E$ open in $\rn$. 
In \cite{BDL} we have studied the borderline case, i.e., $p>1$ and $q=p-1$, 
and in this note we will take on the doubly degenerate case, i.e., $p>2$ and $0<q<p-1$.

Our main result states that locally bounded, weak solutions to \eqref{prototype} are H\"older continuous
in the interior, and up to the parabolic boundary of $E_T$, if proper assumptions on the boundary 
are imposed. Two proofs will be presented, both of which are entirely local and structural.

As a matter of fact, we shall consider  parabolic partial differential equations of the general form
\begin{equation}  \label{Eq:1:1}
	\partial_t\big(|u|^{q-1}u\big)-\dvg\bl{A}(x,t,u, Du) = 0\quad \mbox{ weakly in $ E_T$}
\end{equation}
where $\bl{A}(x,t,u,\z)\colon E_T\times\rr^{N+1}\to\rn$ is a Carath\'eodory function.
Namely, it is 
measurable with respect to $(x, t) \in E_T$ for all $(u,\z)\in \rr\times\rn$,
and continuous with respect to $(u,\z)$ for a.e.~$(x,t)\in E_T$.
Moreover, we assume  the structure conditions
\begin{equation}\label{Eq:1:2p}
	\left\{
	\begin{array}{c}
		\bl{A}(x,t,u,\z)\cdot \z\ge C_o|\z|^p, \\[5pt]
		|\bl{A}(x,t,u,\z)|\le C_1|\z|^{p-1},%
	\end{array}
	\right .
	\quad \mbox{for a.e.~$(x,t)\in E_T$, $\forall\,u\in\rr$, $\forall\,\z\in\rn$,}
\end{equation}
where $C_o$ and $C_1$ are given positive constants. 



In the sequel, we will refer to the set of parameters
$\{p,\,q,\,N,\,C_o,\,C_1\}$ as  the {\it structural data}.
We also write $\boldsymbol \gm$ as a generic positive constant that can be quantitatively
determined a priori only in terms of the data and that can change from line to line.

Postponing the formal definitions of weak solution,
we will proceed to present the main results on the interior regularity in Section~\ref{S:interior}
and the boundary regularity in Section~\ref{S:boundary}.

\subsection{Interior Regularity}\label{S:interior}
Suppose that $u\in L^{\infty}(E_T)$ and set $M:=\|u\|_{\infty, E_T}$.
Let $\Gm:=\pl E_T-\overline{E}\times\{T\}$
be the parabolic boundary of $E_T$. For a compact set $\mathcal{K}\subset E_T$
we introduce the following intrinsic, parabolic distance from $\mathcal{K}$ to $\Gm$ by
\begin{equation*}
	\begin{aligned}
		\dist_p(\mathcal{K};\,\Gm)&:=\inf_{\substack{(x,t)\in \mathcal{K}\\(y,s)\in\Gm}}
		\left\{|x-y|+M^{\frac{p-q-1}p}|t-s|^{\frac1p}\right\}.
	\end{aligned}
\end{equation*}
Now we state our main result concerning the interior H\"older continuity of weak solutions
to \eqref{Eq:1:1}, subject to the structure conditions \eqref{Eq:1:2p}.
Throughout this note, we assume that $p>2$ and $0<q<p-1$ unless otherwise stated.
\begin{theorem}\label{Thm:1:1}
	Let $u$ be a bounded, local, weak solution to \eqref{Eq:1:1} -- \eqref{Eq:1:2p} in $E_T$.
	Then $u$ is locally H\"older continuous in $E_T$. More precisely,
	there exist constants $\boldsymbol\gm>1$ and $\be\in(0,1)$ that can be determined a priori
	only in terms of the data, such that for every compact set $\mathcal{K}\subset E_T$,
	\begin{equation*}
	\big|u(x_1,t_1)-u(x_2,t_2)\big|
	\le
	\boldsymbol \gm M   
	\left(\frac{|x_1-x_2|+M^{\frac{p-q-1}p}|t_1-t_2|^{\frac1p}}{\dist_p(\mathcal{K};\Gm)}\right)^{\be},
	\end{equation*}
for every pair of points $(x_1,t_1), (x_2,t_2)\in \mathcal{K}$.
\end{theorem}
\begin{remark}\label{Rmk:1:1}\upshape
Local boundedness is sufficient for Theorem~\ref{Thm:1:1} to hold.
In fact, local boundedness is inherent in the notion of weak solutions, cf.~Appendix~\ref{Append:1}.
 Moreover, the method also applies to equations with lower order terms like in \cite[Chapters~II -- IV]{DB}
 and in \cite[Appendix~C]{DBGV-mono}.
However we will not pursue generality in this direction. Instead, concentration will be made on
the actual novelty.
\end{remark}
\begin{remark}\upshape
Theorem~\ref{Thm:1:1} implies a Liouville type theorem; the argument is the same as \cite[Corollary~1.1]{BDL} 
which we refer to for details.
\end{remark}

\subsection{Boundary Regularity}\label{S:boundary}
Results on the boundary regularity will be stated in this section.
Let us first consider the following initial-boundary value problem of Dirichlet type:
\begin{equation}\label{Dirichlet}
\left\{
\begin{array}{c}
	\partial_t\big(|u|^{q-1}u\big)-\dvg\bl{A}(x,t,u, Du) = 0\quad \mbox{weakly in $ E_T$,}\\[7pt]
	u(\cdot,t)\Big|_{\partial E}=g(\cdot,t)\Big|_{\partial E}\quad \mbox{for a.e.~$ t\in(0,T]$,}\\[7pt]
	u(\cdot,0)=u_o(\cdot),
\end{array}
\right.
\end{equation}
where the structure conditions \eqref{Eq:1:2p} are in force. 
Concerning the Dirichlet datum $g$ at the lateral boundary $S_T:=\pl E\times(0,T]$ and the initial datum $u_o$ we assume
\begin{align}
\tag{\bf{I}}\label{I} & \mbox{$u_o$ is continuous in $\overline{E}$ with modulus of continuity $\om_{o}(\cdot)$;}\\
\tag{\bf D}\label{D} & \mbox{$\dsty g\in L^p\big(0,T;W^{1,p}( E)\big)$, and $g$ is continuous on $S_T$ with modulus of continuity
  				$\om_g(\cdot)$.} 
\end{align}
As for the geometry of the boundary $\pl E$, we introduce the property of {\it positive geometric density}
\begin{equation}\tag{\bf G}\label{geometry}
\left\{\;\;
	\begin{minipage}[c][1.5cm]{0.7\textwidth}
	there exists $\al_*\in(0,1)$ and $\rho_o>0$, such that for all $x_o\in\pl E$,
	for every cube $K_\rho(x_o)$ and $0<\rho\le\rho_o$, there holds
	$$
	|E\cap K_{\rho}(x_o)|\le(1-\al_*)|K_\rho|.
	$$
	\end{minipage}
\right.
\end{equation}
Here for $\varrho>0$ we have set $K_\varrho(x_o)$ to be the cube with center at $x_o\in\rn$
and edge $2\varrho$, whose faces are parallel with the coordinate planes. When $x_o=0$ we simply write $K_\varrho$.
Intuitively, condition \eqref{geometry} means that there is an exterior cone whose vertex is attached to $x_o$
and whose angle is quantified by $\al_*$.

Next, we consider the Neumann problem. 
The boundary $\pl E$ is assumed to be of class $C^1$, 
such that the outward unit normal, which we denote by {\bf n},
is defined on $\pl E$.
The initial-boundary value problem of Neumann type is formulated as
\begin{equation}\label{Neumann}
\left\{
\begin{array}{c}
	\partial_t\big(|u|^{q-1}u\big)-\dvg\bl{A}(x,t,u, Du) = 0\quad \mbox{weakly in $ E_T$,}\\[5pt]
	\bl{A}(x,t,u, Du)\cdot {\bf n}=\psi(x,t, u)\quad \mbox{on $S_T$,}\\[5pt]
	u(\cdot,0)=u_o(\cdot),
\end{array}
\right.
\end{equation}
where the structure conditions \eqref{Eq:1:2p} and assumption \eqref{I} for the initial data are still in force. 
For the Neumann datum $\psi$ we assume for simplicity that, for some absolute constant $C_2$,
there holds
\begin{equation}\label{N-data}\tag{\bf{N}}
|\psi(x,t, u)|\le C_2\quad \text{ for a.e. }(x,t, u)\in S_T\times\rr.
\end{equation}
Although more general conditions should also work (cf.~ \cite[Section~2, Chapter~II]{DB}),
we however will not pursue generality in this direction.

The formal definitions of weak solutions to \eqref{Dirichlet} and \eqref{Neumann} will be given in Section~\ref{S:1:2}.
Now we are ready to present the results concerning regularity of solutions to \eqref{Dirichlet}
or \eqref{Neumann} up to the parabolic boundary $\Gm$. Recall also that we have set $M:=\|u\|_{\infty, E_T}$.
\subsubsection{Near the Initial Time}
\begin{theorem}\label{Thm:1:2}
Let $u$ be a bounded weak solution to
the Dirichlet problem 
\eqref{Dirichlet} under the assumption \eqref{Eq:1:2p}.
Assume \eqref{I} holds. Then $u$ is continuous in $K\times[0,T]$ for any compact set $K\subset E$.
More precisely, there is a modulus of continuity $\boldsymbol\om(\cdot)$,
determined by the data, $\dist(K,\pl E)$, $M$ and $\boldsymbol\om_{o}(\cdot)$, such that
	\begin{equation*}
	\big|u(x_1,t_1)-u(x_2,t_2)\big|
	\le
	\boldsymbol\om\Big(|x_1-x_2|+M^{\frac{p-q-1}p}|t_1-t_2|^{\frac1p}\Big),
	\end{equation*}
for every pair of points $(x_1,t_1), (x_2,t_2)\in K\times[0,T]$.
In particular, if $u_o$ is H\"older continuous with exponent $\be_{o}$,
then $\boldsymbol\om(r)=\boldsymbol\gm M r^{\be}$ with some $\boldsymbol\gm>0$ and 
$\be\in(0,\be_{o}]$ depending on the data, $\dist(K,\pl E)$ and $\be_{o}$.
\end{theorem}
\begin{remark}\upshape
As we shall see in the proof of Theorem~\ref{Thm:1:2}, 
the estimate on the modulus of continuity actually holds true for all $p>1$ and $q>0$,
if $t_1=0$ or $t_2=0$.
\end{remark}
\subsubsection{Near $S_T$--Dirichlet Type Data}
\begin{theorem}\label{Thm:1:3}
Let $u$ be a bounded weak solution to the Dirichlet problem 
\eqref{Dirichlet} under the assumption \eqref{Eq:1:2p}. Assume \eqref{D} and \eqref{geometry} hold. Then $u$ is continuous in any compact set 
$\mathcal{K}\subset\overline{E}_T$.
More precisely, there is a modulus of continuity $\boldsymbol\om(\cdot)$,
determined by the data, $\al_*$, $\rho_o$, $\dist(\mathcal{K};\{t=0\})$, $M$ and $\boldsymbol\om_{g}(\cdot)$, such that
	\begin{equation*}
	\big|u(x_1,t_1)-u(x_2,t_2)\big|
	\le
	\boldsymbol\om\Big(|x_1-x_2|+M^{\frac{p-q-1}p}|t_1-t_2|^{\frac1p}\Big),
	\end{equation*}
for every pair of points $(x_1,t_1), (x_2,t_2)\in \mathcal{K}$.
In particular, if $g$ is H\"older continuous with exponent $\be_{g}$,
then $\boldsymbol\om(r)=\boldsymbol\gm M r^{\be} $ with some $\boldsymbol\gm>0$ and 
$\be\in(0,\be_{g}]$ depending on the data, $\al_*$, $\rho_o$, $\dist(\mathcal{K};\{t=0\})$ and $\be_{g}$.
\end{theorem}
\subsubsection{Near $S_T$--Neumann Type Data}
\begin{theorem}\label{Thm:1:4}
Let $u$ be a bounded weak solution to the Neumann problem
\eqref{Neumann}.
 Assume $\pl E$ is of class $C^1$ and \eqref{N-data} holds. 
 Then $u$ is H\"older continuous in any compact set 
$\mathcal{K}\subset\overline{E}_T$.
More precisely, there exist constants $\boldsymbol\gm>1$ and $\be\in(0,1)$
determined by the data, $C_2$, $\dist(\mathcal{K};\{t=0\})$ and the structure of $\pl E$, such that
	\begin{equation*}
	\big|u(x_1,t_1)-u(x_2,t_2)\big|
	\le
	\boldsymbol \gm
	M\,
	\left(|x_1-x_2|+M^{\frac{p-q-1}p}|t_1-t_2|^{\frac1p}\right)^\be,
	\end{equation*}
for every pair of points $(x_1,t_1), (x_2,t_2)\in \mathcal{K}$.
\end{theorem}
\begin{remark}\upshape
The proofs of Theorems~\ref{Thm:1:2} -- \ref{Thm:1:4} are local in nature. 
As a result, it suffices to require the boundary 
data in the Dirichlet problem \eqref{Dirichlet}
or the Neumann problem \eqref{Neumann} to be taken just on a portion of the parabolic boundary.
\end{remark}
\subsection{Novelty and Significance}\label{S:NS}
The doubly nonlinear parabolic equation \eqref{Eq:1:1} accounts for many physical models, including dynamics of glaciers, shallow water flows and friction dominated flows in a gas network. We refer to \cite{BDL} for a source of physical motivations.
The mathematical interest of this equation lies in the degeneracy or the singularity or both it possesses,
and a broader class of parabolic equations it generates, which include the parabolic $p$-Laplacian and the porous medium equation as particular instances.

The issue of local H\"older regularity for this equation has been investigated by a number of authors, in various forms and with different notions of solution, cf.~\cite{Henriques-13, Ivanov-89, PV, Zhou-94}. However, all of them assume that $p>2$ and $0<q<1$. 

\begin{figure}[h]\label{Fig:range}
\centering
\includegraphics[scale=0.8]{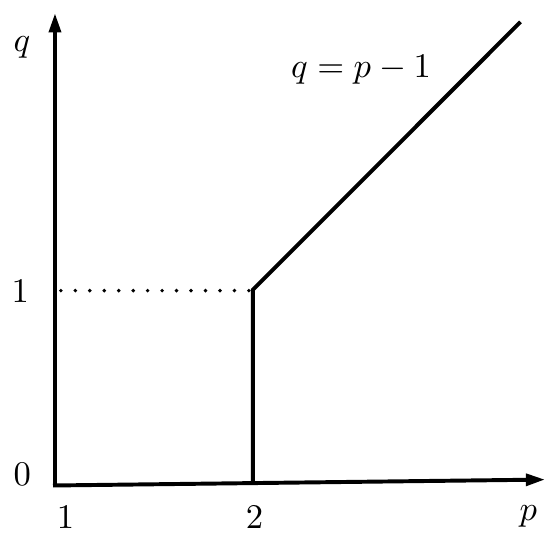}
\caption{Range of $p$ and $q$ }
\end{figure}
The main novelty of our results consists in extending the known range to a larger one, that is, $p>2$ and $0<q<p-1$, cf.~Figure~\ref{Fig:range}. On the other hand, even in the case $p>2$ and $0<q<1$, our results are not covered by the previous works, as they either use different notions of solution \cite{Ivanov-89, PV, Zhou-94}, or assume non-negativity of the solution \cite{Henriques-13}.

One of our main technical advances from the previous works lies in that we dispense with any kind of logarithmic type energy estimates. As such our arguments should have further implications in the context of the so-called $Q$-minima from the calculus of variations, cf.~\cite{Zhou-94}.  

The expansion of positivity for the degenerate parabolic equations has been established in \cite{DBGV-acta}
as a key tool to study Harnack's inequalty.
Roughly speaking, it asserts that the measure of the positivity set of a non-negative, super-solution
translates into pointwise positivity at later times.
Using it to handle the H\"older regularity seems new in the doubly degenerate setting. Similar ideas have appeared in \cite{GSV, Liao} in different forms, either for the parabolic $p$-Laplacian or for the porous medium equation. The virtual advantage of this important property lies in the simplification it brings and a geometric character it offers. On the other hand, the proof of this property is not easy, and meanwhile it is only known to hold in the context of partial differential equations. This latter point unfortunately results in certain restrictions for its application near the boundary.
In particular, when we deal with the boundary regularity for Neumann problems, the original approach of DiBenedetto \cite{DB86}
has to be evoked and adapted.

Our arguments can be adapted to the borderline cases. In particular, when $q=1$, the arguments deal with the degenerate, parabolic $p$-Laplacian; when $p=2$, the porous medium equation can be treated; when $q=p-1$, we are back to our first work \cite{BDL}; see also \cite{Trud1} for non-negative solutions.
Beyond these borderline cases, it will be a subject of our next investigations.
\subsection{Notations and Definitions}\label{S:1:2}
\subsubsection{Notion of Local Solution}\label{S:1:2:1}
A function
\begin{equation}  \label{Eq:1:3p}
	u\in C\big(0,T;L^{q+1}_{\loc}(E)\big)\cap L^p_{\loc}\big(0,T; W^{1,p}_{\loc}(E)\big)
\end{equation}
is a local, weak sub(super)-solution to \eqref{Eq:1:1} with the structure
conditions \eqref{Eq:1:2p}, if for every compact set $K\subset E$ and every sub-interval
$[t_1,t_2]\subset (0,T]$
\begin{equation}  \label{Eq:1:4p}
	\int_K |u|^{q-1}u\z \,\dx\bigg|_{t_1}^{t_2}
	+
	\iint_{K\times (t_1,t_2)} \big[-|u|^{q-1}u\z_t+\bl{A}(x,t,u,Du)\cdot D\z\big]\dx\dt
	\le(\ge)0
\end{equation}
for all non-negative test functions
\begin{equation*}
\z\in W^{1,q+1}_{\loc}\big(0,T;L^{q+1}(K)\big)\cap L^p_{\loc}\big(0,T;W_o^{1,p}(K)%
\big).
\end{equation*}
This guarantees that all the integrals in \eqref{Eq:1:4p} are convergent.

A function $u$ that is both a local weak sub-solution and a local weak super-solution
to \eqref{Eq:1:1} -- \eqref{Eq:1:2p} is a local weak solution.

\subsubsection{Notion of Solution to the Dirichlet Problem}\label{S:1:4:3}
A function
\begin{equation*}  
	u\in C\big(0,T;L^{q+1}(E)\big)\cap L^p\big(0,T; W^{1,p}(E)\big)
\end{equation*}
is a weak sub(super)-solution to \eqref{Dirichlet}, 
if for every sub-interval
$[t_1,t_2]\subset (0,T]$,
\begin{equation*} 
\begin{aligned}
	\int_{E} |u|^{q-1}u\z \,\dx\bigg|_{t_1}^{t_2}
	&+
	\iint_{E\times(t_1,t_2)} \big[-|u|^{q-1}u\z_t+\bl{A}(x,t,u,Du)\cdot D\z\big]\dx\dt
	\le(\ge)0
\end{aligned}
\end{equation*}
for all non-negative test functions
\begin{equation*}
\z\in W_{\loc}^{1,q+1}\big(0,T;L^{q+1}(E)\big)\cap L_{\loc}^p\big(0,T;W_o^{1,p}(E)%
\big).
\end{equation*}
Moreover, setting $\hat{q}:=\min\{2,q+1\}$,
the initial datum is taken in the sense that for any compact set $K\subset E$,
\[
\int_{K\times\{t\}}(u-u_o)^{\hat{q}}_{\pm}\,\dx\to0\quad\text{ as }t\downarrow0.
\]
The Dirichlet datum $g$ is attained under $u\le(\ge)g$
on $\pl E$ in the sense that the traces of $(u-g)_{\pm}$
vanish as functions in $W^{1,p}(E)$ for a.e. $t\in(0,T]$, i.e. $(u-g)_{\pm}\in L^p(0,T; W^{1,p}_o(E))$.
Notice that no {\it a priori} information is assumed on the smoothness of $\pl E$.

A function $u$ that is both a weak sub-solution and a weak super-solution
to \eqref{Dirichlet} is a weak solution.
\subsubsection{Notion of Solution to the Neumann Problem}\label{S:1:4:4}
A function
\begin{equation*}  
	u\in C\big(0,T;L^{q+1}(E)\big)\cap L^p\big(0,T; W^{1,p}(E)\big)
\end{equation*}
is a weak sub(super)-solution to \eqref{Neumann}, 
if for every compact set $K\subset \rr^N$ and every sub-interval
$[t_1,t_2]\subset (0,T]$,
\begin{equation*}  
\begin{aligned}
	\int_{K\cap E} |u|^{q-1}u\z \,\dx\bigg|_{t_1}^{t_2}
	&+
	\iint_{\{K\cap E\}\times(t_1,t_2)} \big[-|u|^{q-1}u\z_t+\bl{A}(x,t,u,Du)\cdot D\z\big]\dx\dt\\
	&\le(\ge)\iint_{\{K\cap\pl E\}\times(t_1,t_2)}\psi(x,t,u)\z\,\d\sig\dt
\end{aligned}
\end{equation*}
for all non-negative test functions
\begin{equation*}
\z\in W_{\loc}^{1,q+1}\big(0,T;L^{q+1}(K)\big)\cap L_{\loc}^p\big(0,T;W_o^{1,p}(K)%
\big).
\end{equation*}
Here $\d\sig$ denotes the surface measure on $\pl E$.
The Neumann datum $\psi$ is reflected in the boundary integral on the right-hand side.
Moreover, the initial datum is taken 
as in the Dirichlet problem.

A function $u$ that is both a weak sub-solution and a weak super-solution
to \eqref{Neumann} is a weak solution.


\medskip
{\it Acknowledgement.} 
V.~B\"ogelein and N.~Liao have been supported by the FWF-Project P31956-N32
``Doubly nonlinear evolution equations".

\section{Energy Estimates}\label{S:energy}

In this section we present certain energy estimates for weak sub(super)-solutions to \eqref{Eq:1:1} -- \eqref{Eq:1:2p}.
They are analogs of the energy estimates   derived in \cite{BDL}, which will be referred to for details.
Moreover, it is noteworthy that they actually hold true for all $p>1$ and $q>0$.

The different roles played by sub-solutions and super-solutions are emphasized.
When we state {\it ``$u$ is a sub(super)-solution...''}
and use $``\pm"$ or $``\mp"$ in what follows, we mean the sub-solution corresponds to
the upper sign and the super-solution corresponds to the lower sign in the statement.

 For any $k\in\rr$, we denote the truncated functions
\[
(u-k)_+\equiv \max\big\{u-k,0\big\},\qquad (u-k)_-\equiv \max\big\{-(u-k),0\big\}.
\]
For $w,k\in\rr$ we define two non-negative quantities
\begin{equation*}
	\mathfrak g_\pm (w,k)=\pm q\int_{k}^{w}|s|^{q-1}(s-k)_\pm\,\ds.
\end{equation*}
For $b\in\rr$ and $\al>0$, we will embolden $\boldsymbol{b}^\al$ to denote the
signed $\al$-power of $b$ as 
\begin{align*}
\boldsymbol{b}^\al=
\left\{
\begin{array}{cl}
|b|^{\al-1}b, &b\neq0,\\[5pt]
0, &b=0.
\end{array}\right.
\end{align*}

Throughout the rest of this note, 
we will use the symbols 
\begin{equation*}
\left\{
\begin{aligned}
(x_o,t_o)+Q_\rho(\theta)&:=K_{\rho}(x_o)\times(t_o-\theta\rho^p,t_o),\\[5pt]
(x_o,t_o)+Q_{R,S}&:=K_R(x_o)\times (t_o-S,t_o),
\end{aligned}\right.
\end{equation*} 
to denote (backward) cylinders with the indicated positive parameters;
when the context is unambiguous, we will omit the vertex $(x_o,t_o)$ from the symbols for simplicity.

First of all, we present energy estimates for local weak sub(super)-solutions defined in Section~\ref{S:1:2:1}.
The proof is similar to \cite[Proposition~3.1]{BDL}, which we refer to for details.
The only difference is that in the present situation, $\boldsymbol{u}^{p-1}$ must be replaced by $\boldsymbol{u}^q$ in terms related to the time derivative and $\mathfrak g_\pm$ has to be defined as above.
Since the testing functions and the treatment of the term containing the vector-field $\bl A$ remain unchanged, the constant $\boldsymbol\gamma$ on the right-hand side of the estimates is independent of $q$.

\begin{proposition}\label{Prop:2:1}
	Let $u$ be a  local weak sub(super)-solution to \eqref{Eq:1:1} -- \eqref{Eq:1:2p} in $E_T$.
	There exists a constant $\boldsymbol \gm (C_o,C_1,p)>0$, such that
 	for all cylinders $Q_{R,S}\Subset E_T$,
 	every $k\in\rr$, and every non-negative, piecewise smooth cutoff function
 	$\z$ vanishing on $\pl K_{R}(x_o)\times (t_o-S,t_o)$,  there holds
\begin{align*}
	\max \bigg\{ \essup_{t_o-S<t<t_o}&\int_{K_R(x_o)\times\{t\}}	
	\z^p\mathfrak g_\pm (u,k)\,\dx,
	\iint_{Q_{R,S}}\z^p|D(u-k)_\pm|^p\,\dx\dt \bigg\}\\
	&\le
	\boldsymbol \gm\iint_{Q_{R,S}}
		\Big[
		(u-k)^{p}_\pm|D\z|^p + \mathfrak g_\pm (u,k)|\partial_t\z^p|
		\Big]
		\,\dx\dt\\
	&\phantom{\le\,}
	+\int_{K_R(x_o)\times \{t_o-S\}} \z^p \mathfrak g_\pm (u,k)\,\dx.
\end{align*}
\end{proposition}

Next, we consider the situation near the initial level $t=0$
when a continuous datum $u_o$ is prescribed.
Suppose the level $k$ satisfies
\begin{equation}\label{Eq:3:2}
	\left\{
	\begin{aligned}
	&k\ge\sup_{K_R(x_o)}u_o\quad\text{ for sub-solutions},\\
	&k\le\inf_{K_R(x_o)}u_o\quad\text{ for super-solutions}.
	\end{aligned}
	\right.
\end{equation}
The following energy estimate can be obtained as in \cite[Proposition~3.2]{BDL}.
\begin{proposition}\label{Prop:2:2}
	Let $u$ be a  local weak sub(super)-solution to \eqref{Dirichlet} with \eqref{Eq:1:2p} in $E_T$.
	There exists a constant $\boldsymbol \gm (C_o,C_1,p)>0$, such that
 	for all cylinders $K_R(x_o)\times (0,S)\subset E_T$,
 	every $k\in\rr$ satisfying \eqref{Eq:3:2}
	and every non-negative, piecewise smooth cutoff function
 	$\z$ independent of $t$ and vanishing on $\pl K_{R}(x_o)$,  there holds
\begin{align*}
	\essup_{0<t<S}&\int_{K_R(x_o)\times\{t\}}	
	\z^p\mathfrak g_\pm (u,k)\,\dx
	+
	\iint_{K_R(x_o)\times (0,S)}\z^p|D(u-k)_\pm|^p\,\dx\dt\\
	&\le
	\boldsymbol \gm\iint_{K_R(x_o)\times (0,S)}
		(u-k)^{p}_\pm|D\z|^p
		\,\dx\dt.
\end{align*}
\end{proposition}

Next, we turn our attention to the energy estimates near $S_T$.
When dealing with Dirichlet data we need to assume
the following restrictions on the level $k$
\begin{equation}\label{Eq:3:3}
	\left\{
	\begin{aligned}
	&k\ge\sup_{Q_{R,S}\cap S_T}g\quad\text{ for sub-solutions},\\
	&k\le\inf_{Q_{R,S}\cap S_T}g\quad\text{ for super-solutions}.
	\end{aligned}
	\right.
\end{equation} 
The following energy estimate can be obtained as in \cite[Proposition~3.3]{BDL}.
\begin{proposition}\label{Prop:2:3}
	Let $u$ be a  local weak sub(super)-solution to \eqref{Dirichlet} with \eqref{Eq:1:2p} in $E_T$.
	There exists a constant $\boldsymbol \gm (C_o,C_1,p)>0$, such that
 	for all cylinders $Q_{R,S} $ with the vertex $(x_o,t_o)\in S_T$,
 	every $k\in\rr$ satisfying \eqref{Eq:3:3}, and every non-negative, piecewise smooth cutoff function
 	$\z$ vanishing on $\pl K_{R}(x_o)\times (t_o-S,t_o)$,  there holds
\begin{align*}
	\max \bigg\{
	\essup_{t_o-S<t<t_o}&\int_{\{K_R(x_o)\cap E\}\times\{t\}}	
	\z^p\mathfrak g_\pm (u,k)\,\dx,
	\iint_{Q_{R,S}\cap E_T}\z^p|D(u-k)_\pm|^p\,\dx\dt
	\bigg\}\\
	&\le
	\boldsymbol \gm\iint_{Q_{R,S}\cap E_T}
		\Big[
		(u-k)^{p}_\pm|D\z|^p + \mathfrak g_\pm (u,k)|\partial_t\z^p|
		\Big]
		\,\dx\dt\\
	&\phantom{\le\ }
	+\int_{\{K_R(x_o)\cap E\}\times \{t_o-S\}} \z^p \mathfrak g_\pm (u,k)\,\dx.
\end{align*}
\end{proposition}

Finally, we deal with the energy estimates for the Neumann problem \eqref{Neumann}.
The following  can be obtained as in \cite[Proposition~3.4]{BDL}.
\begin{proposition}\label{Prop:2:4}
	Let $u$ be a  local weak sub(super)-solution to \eqref{Neumann} with \eqref{Eq:1:2p} in $E_T$.
	Assume $\pl E$ is of class $C^1$ and \eqref{N-data} holds.
	There exists a constant $\boldsymbol \gm >0$ depending on $C_o$, $C_1$, $p$ and the structure of $\pl E$, such that
 	for all cylinders $Q_{R,S} $ with the vertex $(x_o,t_o)\in S_T$,
 	every $k\in\rr$, and every non-negative, piecewise smooth cutoff function
 	$\z$ vanishing on $\pl K_{R}(x_o)\times (t_o-S,t_o)$,  there holds
 \begin{align*}
 	\max \bigg\{
	\essup_{t_o-S<t<t_o}&\int_{\{K_R(x_o)\cap E\}\times\{t\}} \zeta^p \mathfrak g_\pm(u,k)\,\dx,
	\iint_{Q_{R,S}\cap E_T} \zeta^p |D(u-k)_\pm|^p \dx\dt
	\bigg\} \\
	&\le 
	\boldsymbol\gamma \iint_{Q_{R,S}\cap E_T}\Big[ (u-k)^p_\pm|D\zeta|^p 
	+
	\mathfrak g_\pm(u,k)| \partial_t\zeta^p|\Big]\,\dx\dt\\
	&\phantom{\le\,}
	+\boldsymbol\gamma C_2^{\frac{p}{p-1}}\iint_{Q_{R,S}\cap E_T} \z^p\chi_{\{(u-k)_{\pm}>0\}}\,\dx\dt\\
	&\phantom{\le\,}
	+
	\int_{\{K_R(x_o)\cap E\}\times\{t_o-S\}} \zeta^p \mathfrak g_\pm(u,k)\,\dx.
\end{align*}
\end{proposition}

\section{Preliminary Tools} 
For a compact set $K\subset\rr^N$ and
 a cylinder $\mathcal{Q}:=K\times(T_1,T_2]\subset E_T$
we introduce numbers $\boldsymbol\mu^{\pm}$ and $\boldsymbol\om$ satisfying
\begin{equation*}
	\boldsymbol\mu^+\ge\essup_{\mathcal{Q}}u,
	\quad 
	\boldsymbol\mu^-\le\essinf_{\mathcal{Q}} u,
	\quad
	\boldsymbol\om\ge\boldsymbol\mu^+-\boldsymbol\mu^-.
\end{equation*}

In this section, we collect some lemmas, which will be the main ingredients
in the proof of Theorem~\ref{Thm:1:1}. 
The first one is a De Giorgi type lemma, which actually holds for all $p>1$ and $q>0$.

\begin{lemma}\label{Lm:DG:1}
 Let $u$ be a locally bounded, local weak sub(super)-solution to \eqref{Eq:1:1} -- \eqref{Eq:1:2p} in $E_T$.
 Set $\theta=(\xi\boldsymbol\om)^{q+1-p}$ for some $\xi\in(0,1)$ and assume $(x_o,t_o) + Q_\varrho(\theta) \subset \mathcal{Q}$.
There exists a constant $\nu\in(0,1)$ depending only on 
 the data, such that if
\begin{equation*}
	\Big|\Big\{
	\pm\big(\boldsymbol \mu^{\pm}-u\big)\le \xi\boldsymbol\om\Big\}\cap (x_o,t_o)+Q_{\varrho}(\theta)\Big|
	\le
	\nu|Q_{\varrho}(\theta)|,
\end{equation*}
then either
\begin{equation*}
	|\boldsymbol\mu^{\pm}|>8\xi\boldsymbol\om,
\end{equation*}
or
\begin{equation*}
	\pm\big(\boldsymbol\mu^{\pm}-u\big)\ge\tfrac{1}2\xi\boldsymbol\om
	\quad
	\mbox{a.e.~in $(x_o,t_o)+Q_{\frac{1}2\varrho}(\theta)$.}
\end{equation*}
\end{lemma}
\begin{proof}
The De Giorgi iteration has been performed in \cite[Lemma~2.2]{Liao-21} for super-solutions, whereas the proof for sub-solutions is analogous.
In order to obtain the present formulation, choose $a=\frac{1}{2}$ and replace $M$ by $\xi \boldsymbol \omega$.
If $| \boldsymbol \mu^\pm | > 8 \xi \omega$, there is nothing to prove.
In the opposite case, the assumption $| \boldsymbol \mu^\pm | \leq 8 \xi \omega$ allows us to estimate $\max\{ L^{q-1}, M^{q-1} \}$ by $\max\{9^{q-1}, 1\} M^{q-1}$.
Therefore, the critical number $\nu$ depends only on the data.
%
\end{proof}

The next lemma is a variant of the previous one, involving quantitative initial data.
Again, it actually holds for all $p>1$ and $q>0$
\begin{lemma}\label{Lm:DG:initial:1}
Let $u$ be a locally bounded, local weak sub(super)-solution to \eqref{Eq:1:1} -- \eqref{Eq:1:2p} in $E_T$. 
Set $\theta=(\xi\boldsymbol\om)^{q+1-p}$ for some $\xi\in(0,1)$. 
There exists a positive constant $\nu_o$ depending only on the data, such that if
\[
\pm\big(\boldsymbol\mu^{\pm}-u(\cdot, t_o)\big)\ge \xi\boldsymbol\om,\quad\text{ a.e. in } K_{\rho}(x_o),
\]
then either
\begin{equation*}
	|\boldsymbol\mu^{\pm}|>8\xi\boldsymbol\om,
\end{equation*}
or
\[
\pm\big(\boldsymbol\mu^{\pm}-u\big)\ge \tfrac12\xi\boldsymbol\om\quad\text{ a.e. in }K_{\frac12\rho}(x_o)\times(t_o,t_o+\nu_o\theta\rho^p],
\]
provided the cylinders are included in $\mathcal{Q}$.
\end{lemma}
\begin{proof}
After enforcing that $|\boldsymbol\mu^{-}|\le 8\xi\boldsymbol\om$,
this is essentially the content of \cite[Lemma~3.1]{Liao-21} for super-solutions, the case of sub-solutions being similar. 
More precisely, one has to choose $a=\frac{1}{2}$, replace $ M $ by $\xi\boldsymbol\omega$ and note that the constant $\max\{L^{q-1},M^{q-1}\}$ is controlled by
$\max\{ 1, 9^{q-1} \}(\xi\boldsymbol\omega)^{q-1}$ 
whenever $|\boldsymbol\mu^{-}|\le 8\xi\boldsymbol\om$. 
This allows to choose the parameter $\theta$ in \cite[Lemma~3.1]{Liao-21} in the form $\nu_o (\xi\boldsymbol\om)^{q+1-p}$ for some $\nu_o$ depending only on the data.
\end{proof}
The previous lemma propagates pointwise information in a smaller cube, without a time lag.
The next lemma translates measure theoretical information into a pointwise estimate over an expanded cube of later times.
This  is essentially the expansion of positivity for 
the degenerate, parabolic $p$-Laplacian established in \cite{DBGV-acta}; see also  \cite[Chapter~4, Proposition~4.1]{DBGV-mono}. 
As such it actually holds for  $p>2$ and $q>0$.
\begin{lemma}\label{Lm:expansion:p}
Let $u$ be a locally bounded, local, weak sub(super)-solution to \eqref{Eq:1:1} -- \eqref{Eq:1:2p} in $E_T$.
	Introduce the parameters $\Lm,\,c>0$ and $\al\in(0,1)$.
	Suppose that 
	\[
	c\boldsymbol\om\le\pm\boldsymbol \mu^{\pm}\le\Lm\boldsymbol\om
	\]
	and for some $0<a \leq \frac12c$,
	\begin{equation*}
		\Big|\Big\{\pm\big(\boldsymbol \mu^{\pm}-u(\cdot, t_o)\big)\ge a\boldsymbol\om\Big\}\cap K_{\varrho}(x_o)\Big|
		\ge
		\al |K_\varrho|.
	\end{equation*}
There exist constants $b>0$ and $\eta\in(0,1)$ depending only on the data, $\Lm$,  $c$, $a$ and $\al$, such that 
\begin{equation*}
	\pm\big(\boldsymbol \mu^{\pm}-u\big)\ge \eta \boldsymbol\om
	\quad
	\mbox{a.e.~in $K_{2\varrho}(x_o)\times\left(t_o+\tfrac12 b \boldsymbol\om^{q-1}(\eta\boldsymbol\om)^{2-p}\varrho^p,
	t_o+b\boldsymbol\om^{q-1}(\eta\boldsymbol\om)^{2-p}\varrho^p\right],$}
\end{equation*}
provided this cylinder is included in $\mathcal{Q}$.
\end{lemma}
\begin{proof}
We may assume $(x_o,t_o)=(0,0)$ and prove the case of super-solutions only as the other case is similar.
Let $k=\boldsymbol\mu^-+\frac12c\boldsymbol\om$.
By Lemma~\ref{Lm:A:1}, $u_k:=\min\{u, k\}=k-(u-k)_-$ is a local, weak super-solution to \eqref{Eq:1:1}, i.e.
\[
\pl_t \boldsymbol{u_k}^q-\dvg {\bf A}(x,t,u_k,Du_k)\ge0\quad\text{ weakly in }\mathcal{Q}.
\]
Here the symbol $\boldsymbol{u_k}^q$ has been emboldened to denote the signed power of $u_k$ defined in Section~\ref{S:energy}.
To proceed, we define
\[
v:=\boldsymbol{u_k}^q-\boldsymbol{(\mu^-)}^q,
\]
which is non-negative in $\mathcal{Q}$.
Thanks to the restriction on $\boldsymbol\mu^-$, it is not hard to show that
$v$ belongs to the function space \eqref{Eq:1:3p}$_{q=1}$ defined on $\mathcal{Q}$
 and satifies  
\[
v_t-\dvg{\bf \bar{A}}(x,t,v,Dv)\ge0\quad\text{ weakly in }\mathcal{Q}.
\]
Here ${\bf \bar{A}}$ is defined by
\[
	{\bf \bar{A}}(x,t,y,\zeta)
	:=
	{\bf A}\Big(x,t,\big| \widetilde{y} + \boldsymbol{(\mu^-)}^q\big|^{\frac{1-q}q}\big( \widetilde{y} + \boldsymbol{(\mu^-)}^q\big),
	\tfrac1q \big| \widetilde{y}+\boldsymbol{(\mu^-)}^q\big|^\frac{1-q}q \zeta\Big),
\]
where $\widetilde{y}$ denotes the truncation
$$
	\widetilde{y}
	:=
	\min\big\{ \max\{ y, 0 \}, \big(1-\tfrac{1}{2^q}\big) (c\boldsymbol\omega)^q \big\}.
$$
Meanwhile, one verifies that the structure conditions
\begin{equation*}
	\left\{
	\begin{array}{c}
		{\bf \bar{A}}(x,t,y,\zeta)\cdot \zeta\ge\overline C_o\boldsymbol\omega^{(q-1)(1-p)}|\zeta|^p \\[5pt]
		|{\bf \bar{A}}(x,t,y,\zeta)|\le \overline{C}_1 \boldsymbol\omega^{(q-1)(1-p)} |\zeta|^{p-1},%
	\end{array}
	\right .
\end{equation*}
with positive constants $\overline C_i =\overline C_i(C_i,p,q,c,\Lambda),\, i=0,1$.

In order to eliminate the dependence on $\boldsymbol \omega$ in the structure conditions of ${\bf \bar{A}}$, we consider the transformed function
\[
\widetilde{v}(x,t):=v\big(x,\boldsymbol\om^{(q-1)(p-1)}t\big),
\]
which satisfies
\begin{equation}\label{PDE-tilde}
\widetilde{v}_t-\dvg{\bf \widetilde{A}}(x,t,\widetilde{v},D\widetilde{v}) \ge 0
\quad \text{ weakly in }
K \times \big(\boldsymbol\om^{(q-1)(1-p)} T_1, \boldsymbol\om^{(q-1)(1-p)} T_2\big].
\end{equation}
Here ${\bf \widetilde{A}}$ is defined by
$$
	{\bf \widetilde{A}}(x, t, y , \zeta)
	:=
	\boldsymbol\om^{(q-1)(p-1)} {\bf \bar{A}}(x, \boldsymbol\om^{(q-1)(p-1)} t, y , \zeta)
$$
for $(x, t) \in \widetilde{\mathcal Q}:= K \times (\boldsymbol\om^{(q-1)(1-p)} T_1, \boldsymbol\om^{(q-1)(1-p)} T_2]$ and all $(y,\zeta) \in \rr \times \rn$.
Thus, an easy calculation shows that ${\bf \widetilde{A}}$ satisfies the conditions
\begin{equation*}
	\left\{
	\begin{array}{c}
		{\bf \widetilde{A}}(x,t,y,\zeta)\cdot \zeta\ge \overline{C}_o|\zeta|^p \\[5pt]
		|{\bf \widetilde{A}}(x,t,y,\zeta)|\le \overline{C}_1|\zeta|^{p-1}.%
	\end{array}
	\right .
\end{equation*}
In other words, the function $\widetilde{v}$ is a non-negative, local, weak super-solution
to the parabolic $p$-Laplacian type equation \eqref{PDE-tilde} in $\widetilde{\mathcal Q}$. This allows us to apply the expansion of positivity
in  \cite[Chapter~4, Proposition~4.1]{DBGV-mono}. The measure theoretical information for $u$
implies a similar inequality for $u_k$; in fact we have
$$
		\Big|\Big\{u_k(\cdot, 0)\ge \boldsymbol \mu^{-}+ a\boldsymbol\om\Big\}\cap K_{\varrho}\Big|
				\ge
		\al |K_\varrho|.
$$
Taking into account $-\Lambda\boldsymbol\omega\le u_k\le -\frac12 c\boldsymbol\omega$, the information that
$u_k(\cdot, 0)\ge \boldsymbol \mu^{-}+ a\boldsymbol\om$ can be converted into an estimate from below for $v$. Indeed,
by the  mean value theorem, we estimate
\begin{align*}
	v(\cdot ,0)&= \boldsymbol{u_k}^q(\cdot ,0)-(\boldsymbol \mu^-)^q
	\ge
	q\min\big\{\Lm^{q-1}, (\tfrac12c)^{q-1}\big\}\boldsymbol\omega^{q-1}\big( u_k(\cdot ,0)-\boldsymbol\mu^-)\\
	&\ge 
	aq\min\big\{\Lm^{q-1}, (\tfrac12c)^{q-1}\big\}\boldsymbol\omega^q
	=:
	\widetilde a\boldsymbol\omega^q,
\end{align*}
In terms of $\widetilde{v}$ this becomes
	\begin{equation*}
		\Big|\Big\{\widetilde{v}(\cdot, 0)\ge \widetilde{a}\boldsymbol\om^q\Big\}\cap K_{\varrho}\Big|
		\ge
		\al |K_\varrho|.
	\end{equation*}
An application of  \cite[Chapter~4, Proposition~4.1]{DBGV-mono} to $\widetilde{v}$ (with $C\equiv 0$ and $M=\widetilde a\boldsymbol\om^q$)
yields that for some positive constants $\eta,\delta\in (0,1)$ and $b>1$ depending only on the data $\overline C_o, \overline C_1, p, N$
and on $\al$, such that 
\begin{equation*}
	\widetilde v(\cdot ,t)\ge \eta\widetilde a \boldsymbol \om^q
	\quad\mbox{a.e.~in $K_{2\rho}$}
\end{equation*}
for all
\begin{equation*}
	\frac{b^{p-2}}{2 (\eta\widetilde a \boldsymbol\om^q)^{p-2}} \delta\rho^p
	< t\le
	\frac{b^{p-2}}{(\eta\widetilde a \boldsymbol\om^q)^{p-2}}\delta\rho^p.
\end{equation*}
For $v$ this means that 
\begin{equation*}
	v(\cdot ,t)\ge \eta\widetilde a \boldsymbol\om^q
	\quad\mbox{a.e.~in $K_{2\rho}$}
\end{equation*}
for all $t$ in the interval
\begin{equation*}
	\tfrac12 b^{p-2} \delta (\eta \widetilde{a} \boldsymbol\om)^{2-p} \boldsymbol\om^{q-1} \rho^p
	< t\le
	b^{p-2} \delta (\eta \widetilde{a} \boldsymbol\om)^{2-p} \boldsymbol\om^{q-1}\rho^p.
\end{equation*}
We revert to the original function $u$ with the aid of  the mean value theorem. More precisely, we estimate
\begin{align*}
	\eta\widetilde a \boldsymbol\om^q
	&\le v =\boldsymbol{u_k}^q-(\boldsymbol \mu^-)^q
	\le q\max\big\{\Lm^{q-1}, (\tfrac12c)^{q-1}\big\}\boldsymbol\omega^{q-1} \big(u_k-\boldsymbol \mu^-\big)
	\le\widetilde{\boldsymbol\gamma} \boldsymbol\om^{q-1}\big(u-\boldsymbol \mu^-\big)
\end{align*}
for some positive $\widetilde{\boldsymbol\gamma} = \widetilde{\boldsymbol\gamma}(q,c,\Lambda)$. This, however, is equivalent to
\begin{equation*}
	u(\cdot ,t)\ge \boldsymbol \mu^-+\frac{\eta\widetilde a}{\widetilde{\boldsymbol\gamma}} \,\boldsymbol\om
	\quad\mbox{a.e.~in $K_{2\rho}$}
\end{equation*}
for all $t$ in the above interval.
Redefining $\eta\widetilde a/\widetilde{\boldsymbol\gamma}$ as $\eta$ and $\widetilde{\boldsymbol\gamma}^{2-p} b^{p-2}\delta$
as $b$, the claim follows.
\end{proof}
\begin{remark}
\label{rem:smaller eta}\upshape
An inspection of the above proof shows that $\eta=\boldsymbol\gm a$ for some positive $\boldsymbol\gm$ depending only on the data, $\al$, $c$ and $\Lm$.
The conclusion of Lemma \ref{Lm:expansion:p} holds true for a smaller $ \eta$ by properly making $a$ smaller.
\end{remark}
%
The following lemma examines the situation when pointwise information is given at the initial level.
It actually holds for $p>2$ and $q>0$.
\begin{lemma}\label{Lm:DG:initial:2}
Let $u$ be a locally bounded, local weak sub(super)-solution to \eqref{Eq:1:1} -- \eqref{Eq:1:2p} in $E_T$. 
Introduce the parameters $\Lm,\,c>0$ and $\eta\in(0,1)$,
and set $\theta=\boldsymbol\om^{q-1}(\eta\boldsymbol\om)^{2-p}$. 
	Suppose that 
	\[
	c\boldsymbol\om\le\pm\boldsymbol \mu^{\pm}\le\Lm\boldsymbol\om.
	\]
There exists a positive constant $\nu_1$ depending only on the data, $c$ and $\Lm$, such that if
\[
\pm\big(\boldsymbol\mu^{\pm}-u(\cdot, t_o)\big)\ge \eta\boldsymbol\om,\quad\text{ a.e. in } K_{\rho}(x_o),
\]
then 
\[
\pm\big(\boldsymbol\mu^{\pm}-u\big)\ge \tfrac12\eta\boldsymbol\om\quad\text{ a.e. in }K_{\frac12\rho}(x_o)\times(t_o,t_o+\nu_1\theta\rho^p],
\]
provided the cylinders are included in $\mathcal{Q}$.
\end{lemma}
\begin{proof}
Suppose $u$ is a local, weak super-solution as the other case is similar. Moreover, we may assume $(x_o,t_o)=(0,0)$. 
Introduce $\widetilde{v}$ like in the proof of Lemma~\ref{Lm:expansion:p}, which turns out to be a non-negative, local, weak super-solution
to the parabolic $p$-Laplacian type equation \eqref{PDE-tilde} in $\widetilde{\mathcal Q}$.
Using the mean value theorem, the information that $u(\cdot ,0)\ge \boldsymbol\mu^-+\eta\boldsymbol\om$ in $K_\rho$ yields that
$\widetilde v(\cdot, 0)\ge \boldsymbol\gamma \eta\boldsymbol\om^q$ in $K_\rho$ for some positive 
$\boldsymbol\gamma = \boldsymbol\gamma(q, c, \Lm)$.
Consequently, we may apply \cite[Chapter~3, Lemma~4.1]{DBGV-mono} or \cite[Lemma~3.2]{Liao-21} to $\widetilde{v}$. For $a\in (0,1)$ at our disposal we have
\begin{equation*}
	\widetilde v\ge a \boldsymbol\gamma \eta\boldsymbol\om^q
	\quad
	\mbox{a.e.~on $K_{\frac12\rho}\times \big(0, \vartheta \big(\frac12\rho)^p\big]$,}
\end{equation*}
where
$$
	\vartheta = \bar c (1-a)^{N+3}\big(\boldsymbol\gamma\eta\boldsymbol\om^q\big)^{2-p},
$$
for some constant $\bar c\in (0,1)$ depending on $\overline C_o,\overline C_1, p,N$.
As in  the proof of Lemma~\ref{Lm:expansion:p} we convert this into an estimate for $u$. First, the scaling in time
gives
$$
	\frac{\vartheta}{2^p\boldsymbol\om^{(q-1)(1-p)}}
	=2^{-p}\bar c (1-a)^{N+3} \boldsymbol\gamma^{2-p}
	\boldsymbol\omega^{q-1}(\eta\boldsymbol\omega)^{2-p}
	 = 2^{-p}\bar c (1-a)^{N+3} \boldsymbol\gamma^{2-p}\theta=:\nu_1\theta,
$$
so that $v\ge a \boldsymbol\gamma \eta\boldsymbol\om^q$ on $K_{\frac12\rho}\times \big(0, \nu_1\theta \rho^p\big]$.
Note that $\nu_1$ depends on $\overline C_o,\overline C_1, p,q,N,c,\Lambda$ and $a$.
As in the proof of Lemma~\ref{Lm:expansion:p}, we may apply  the mean value theorem to estimate
$$
	a \boldsymbol\gamma \eta \boldsymbol\om^q
	\leq v \leq
	\widetilde{\boldsymbol\gamma} \boldsymbol{\omega}^{q-1}
	\big( u - \boldsymbol\mu^-\big)
$$
for some positive $\widetilde{\boldsymbol\gamma} = \widetilde{\boldsymbol\gamma}(q, c, \Lm)$ and therefore on $K_{\frac12\rho}\times \big(0, \nu_1\theta \rho^p\big]$ there holds
\begin{align*}
	u\ge \boldsymbol\mu^-+ \frac{a\boldsymbol\gamma}{\widetilde{\boldsymbol\gamma}}\eta\boldsymbol\om.
\end{align*}

Finally, choosing the free parameter $a$ such that $a\boldsymbol\gamma/\widetilde{\boldsymbol\gamma}=1/2$ on the one hand determines the value of  $\nu_1$ in dependence on the data, $c$ and $\Lambda$, and on the other hand
implies the desired bound from below. 
\end{proof}
\section{The First Proof of Theorem \ref{Thm:1:1}}\label{S:5}
The proof of Theorem \ref{Thm:1:1} in this section relies on the expansion of positivity from Lemma~\ref{Lm:expansion:p}.
This important tool simplifies our arguments, though the attainment of it is difficult and 
turned out to be a major achievement in the recent theory, cf.~\cite{DBGV-acta, DBGV-mono}.
As such the same simplification can be carried out in \cite{BDL}.
On the other hand, the argument of this section does not seem applicable directly to the boundary regularity for the Neumann problem.
For this reason, we will give a second proof of Theorem \ref{Thm:1:1} in Section~\ref{S:6},
referring back to our previous arguments in \cite{BDL} that are modeled on \cite{DB86}.

\subsection{The Proof Begins}
\label{S:5:1}
Assume $(x_o,t_o)=(0,0)$,
introduce $Q_o=K_\rho\times(-\rho^{p-1},0]\Subset E_T$ with a radius $\rho \leq 1$ and
set
\begin{equation*}
	\boldsymbol \mu^+=\essup_{Q_o}u,
	\quad
	\boldsymbol\mu^-=\essinf_{Q_o}u,
	\quad
	\boldsymbol\om\ge\boldsymbol\mu^+-\boldsymbol\mu^-.
\end{equation*}
Let $\theta=(\frac14\boldsymbol\om)^{q+1-p}$. For some $A>1$ to be determined in terms of the data,
we may assume that
\begin{equation}\label{Eq:start-cylinder}
Q_\rho(A\theta)\subset Q_o,\quad\text{ such that }\quad \essosc_{Q_{\rho}(A\theta)}u\le\boldsymbol\om;
\end{equation}
otherwise we would have 
\begin{equation}\label{Eq:extra-control}
\boldsymbol\om\le L\rho^{\frac1{p-q-1}}\quad\text{ where }L=4A^{\frac1{p-q-1}}.
\end{equation}

Our proof unfolds along two main
cases, namely for some $\xi\in(0,1)$ to be determined,
\begin{equation}\label{Eq:Hp-main}
\left\{
\begin{array}{c}
\mbox{when $u$ is near zero:  $\boldsymbol\mu^-\le\xi\boldsymbol\om$ and 
$\boldsymbol\mu^+\ge-\xi\boldsymbol\om$};\\[5pt]
\mbox{when $u$ is away from zero: $\boldsymbol\mu^->\xi\boldsymbol\om$ or $\boldsymbol\mu^+<-\xi\boldsymbol\om$.}
\end{array}\right.
\end{equation}
Note that \eqref{Eq:Hp-main}$_1$ implies that 
$|\boldsymbol\mu^\pm|\le2\boldsymbol\om$.
We deal with this case in Sections~\ref{S:case-1} -- \ref{S:case-1-3};
the other case will be treated in Section~\ref{S:case-2}.
\subsection{Reduction of Oscillation Near Zero--Part I}\label{S:case-1}
In this section, we will assume that \eqref{Eq:Hp-main}$_1$ holds true.
We work with $u$ as a super-solution near its
infimum.
To proceed further, we assume 
\begin{equation}\label{Eq:mu-pm-}
	\boldsymbol\mu^+ -\boldsymbol\mu^- >\tfrac12\boldsymbol\om.
\end{equation}
The other case $\boldsymbol\mu^+ -\boldsymbol\mu^- \le\frac12\boldsymbol\om$,
will be considered later.
Observe that \eqref{Eq:mu-pm-} implies 
\begin{equation}\label{Eq:mu-pm}
	\boldsymbol\mu^+ - \tfrac18 \boldsymbol\om \ge \tfrac18 \boldsymbol\om
	\quad\text{ or }\quad
	\boldsymbol\mu^- + \tfrac18 \boldsymbol\om\le -\tfrac18 \boldsymbol\om.
\end{equation}
Let us consider for instance the first case, i.e. \eqref{Eq:mu-pm}$_1$, as the other one can be treated analogously.
Hence we have $\frac14 \boldsymbol\om \le \boldsymbol\mu^+ \le 2\boldsymbol\om$ and Lemma~\ref{Lm:expansion:p} is at our disposal with $c=\frac{1}{4}$ and $\Lambda = 2$.

Suppose $A$ is a large number, and consider the ``bottom" sub-cylinder of $Q_\rho(A\theta)$, that is,
$$\widetilde{Q}:=K_{\rho}\times\big(-A\theta\rho^p,-(A-1)\theta\rho^p\big].$$
One of the following two alternatives must hold true:
\begin{equation}\label{Eq:alternative-int}
\left\{
\begin{array}{cc}
\Big|\Big\{u \le \boldsymbol \mu^- + \tfrac14 \boldsymbol\om\Big\}\cap \widetilde{Q}\Big|\le \nu|\widetilde{Q}|,\\[5pt]
\Big|\Big\{u \le \boldsymbol \mu^- + \tfrac14 \boldsymbol\om\Big\}\cap \widetilde{Q}\Big|> \nu |\widetilde{Q}|.
\end{array}\right.
\end{equation}
Here the number $\nu\in(0,1)$ is determined in Lemma~\ref{Lm:DG:1} in terms of the data.

First suppose \eqref{Eq:alternative-int}$_1$ holds true. An application of Lemma~\ref{Lm:DG:1} (with $\xi=\frac14$)
gives us that, recalling $|\boldsymbol\mu^-|\le2\boldsymbol\om$ due to \eqref{Eq:Hp-main}$_1$,
\begin{equation}\label{Eq:initial-condi}
u\ge \boldsymbol\mu^-+\tfrac18\boldsymbol\om\quad\text{ a.e. in }\tfrac12\widetilde{Q}.
\end{equation}
Here the notation $\tfrac12\widetilde{Q}$ should be self-explanatory in view of Lemma~\ref{Lm:DG:1}.
In particular, the above pointwise lower bound of $u$ holds at the time level $t_o=-(A-1)\theta\rho^p$ for a.e.~$x \in K_{\frac{1}{2} \rho}$,
which serves as the initial datum for an application of Lemma~\ref{Lm:DG:initial:1}.
Indeed, we fix $\nu_o$ in Lemma~\ref{Lm:DG:initial:1} depending on the data and choose $\xi \in \big( 0, \frac{1}{8} \big)$ so small that
$$
	0 \leq
	t_o + \nu_o (\xi \boldsymbol \omega)^{q+1-p} \big( \tfrac{1}{2}\rho \big)^p
	=
	- (A-1) \big( \tfrac14 \boldsymbol \omega \big)^{q+1-p} \rho^p
	+ \nu_o (\xi \boldsymbol \omega)^{q+1-p} \big(\tfrac{1}{2}\rho \big)^p,
$$
i.e.~we choose
\begin{equation}\label{Eq:choice-xi}
	\xi = \min \Big\{ \tfrac{1}{8}\,, \,\tfrac14 \Big( \frac{\nu_o}{2^p A} \Big)^{\frac1{p-q-1}} \Big\}.
\end{equation}
Thus, enforcing $|\boldsymbol\mu^-|\le\xi\boldsymbol\om$, we obtain that
\[
u\ge\boldsymbol\mu^-+\tfrac12\xi\boldsymbol\om\quad\text{ a.e. in }K_{\frac14\rho}\times (t_o,0],
\]
which in turn yields a reduction of oscillation
\begin{equation}\label{Eq:reduc-osc-int-1}
\essosc_{Q_{\frac14\rho}(\theta)}u\le \big( 1-\tfrac12\xi \big)\boldsymbol\om.
\end{equation}
Here in \eqref{Eq:choice-xi} we have tacitly used the fact that $q<p-1$ in the determination of $\xi$.
Keep also in mind that $A>1$ is yet to be determined in terms of the data.

The case $\boldsymbol\mu^->\xi\boldsymbol\om$ will be treated in Section~\ref{S:case-2};
whereas if $-2\boldsymbol\om<\boldsymbol\mu^-<-\xi\boldsymbol\om$, 
we may apply Lemma~\ref{Lm:DG:initial:2}
with $c=\xi$, $\Lambda=2$ and $\eta = \eta_o \in \big( 0,\frac{1}{8} \big)$.
Indeed, fixing $\nu_1$ in Lemma~\ref{Lm:DG:initial:2} depending on the data and $\xi$, we choose $\eta_o$ to satisfy 
$$
	\nu_1 \boldsymbol\om^{q-1}(\eta_o\boldsymbol\om)^{2-p} \big( \tfrac{1}{2}\rho \big)^p
	\ge
	A \big( \tfrac14 \boldsymbol\om \big)^{q+1-p} \rho^p,
	\quad \text{ i.e. } \quad
	\eta_o = \min \Big\{ \tfrac{1}{8}, 4^{\frac{p-q-1}{p-2}} \Big(\frac{\nu_1}{2^p A} \Big)^{\frac1{p-2}} \Big\}.
$$
Here we have tacitly used the fact that $p>2$ in the determination of $\eta_o$.
In this way, Lemma~\ref{Lm:DG:initial:2} asserts that
\[
u\ge\boldsymbol\mu^-+\tfrac12\eta_o\boldsymbol\om\quad\text{ a.e. in }K_{\frac14\rho}\times (t_o,0],
\]
which yields  the reduction of oscillation
\begin{equation}\label{Eq:reduc-osc-int-3}
\essosc_{Q_{\frac14\rho}(\theta)}u\le \big( 1-\tfrac12\eta_o \big)\boldsymbol\om.
\end{equation}
\subsection{Reduction of Oscillation Near Zero--Part II}\label{S:case-1-2}
In this section, we still assume that \eqref{Eq:Hp-main}$_1$ and \eqref{Eq:mu-pm-} hold true.
However, we turn our attention to the second alternative \eqref{Eq:alternative-int}$_2$.
We work with $u$ as a sub-solution near its
supremum.
Since under our assumptions there holds $\boldsymbol\mu^+-\frac14\boldsymbol\om\ge\boldsymbol\mu^-+\frac14\boldsymbol\om$, we may rephrase \eqref{Eq:alternative-int}$_2$ as
\[
\Big|\Big\{\boldsymbol\mu^+-u\ge\tfrac14\boldsymbol\om\Big\}\cap \widetilde{Q}\Big|> \nu |\widetilde{Q}|.
\]
Then it is not hard to see that there exists 
$$
	t_*\in \big[-A\theta\rho^p, -(A-1)\theta\rho^p- \tfrac12 \nu\theta\rho^p\big],
$$ 
such that
\[
\Big|\Big\{\boldsymbol\mu^+-u(\cdot, t_*)\ge\tfrac14\boldsymbol\om\Big\}\cap K_\rho\Big|>\tfrac12\nu |K_\rho|.
\]
Otherwise $t_*$ can be in $\big[-A\theta\rho^p, -(A-1)\theta\rho^p\big]$.
Indeed, if the above inequality does not hold for any $s$ in the given 
interval, then 
\begin{align*}
	\Big|\Big\{\boldsymbol\mu^+-u\ge\tfrac14\boldsymbol\om\Big\}\cap 
	\widetilde{Q}\Big|
	&=
	\int_{-A\theta\varrho^p}^{-(A-1)\theta\rho^p-\frac12\nu\theta\varrho^p}
	\Big|\Big\{\boldsymbol\mu^+-u(\cdot, s)\ge \tfrac{1}4\boldsymbol\om\Big\}
	\cap K_{\varrho}\Big|\,\ds\\
	&\phantom{=\,}
	+\int^{-(A-1)\theta\rho^p}_{-(A-1)\theta\rho^p-\frac12\nu\theta\varrho^p}
	\Big|\Big\{\boldsymbol\mu^+-u(\cdot, s)\ge \tfrac{1}4\boldsymbol\om\Big\}
	\cap K_{\varrho}\Big|\,\ds\\
	&<\tfrac12\nu |K_{\varrho}|\theta\varrho^p\big(1-\tfrac12\nu\big) +\tfrac12\nu\theta\varrho^p
	|K_\varrho |
	<\nu |Q_\varrho|,
\end{align*}
implying a contradiction to the above measure theoretical information.
Recall that due to \eqref{Eq:mu-pm}$_1$ we actually have
$\frac14 \boldsymbol\om\le\boldsymbol\mu^+\le 2\boldsymbol\om$.
Then, based on the above measure theoretical information at $t_*$, an application of Lemma~\ref{Lm:expansion:p} with $c = \frac{1}{4}$, $\Lambda = 2$ and a fixed constant $a = \frac{1}{8} = \frac{1}{2} c$ yields constants $b>0$ and $\eta_1\in(0,1)$ depending only on the data, such that
\[
\boldsymbol \mu^+-u(\cdot, t)\ge\eta_1\boldsymbol\om\quad\text{ a.e. in }K_{\frac12\rho}
\]
for all times
\[
	t_*+\tfrac12b\boldsymbol\om^{q-1}(\eta_1\boldsymbol\om)^{2-p}\rho^p
	\le t\le
	t_*+b\boldsymbol\om^{q-1}(\eta_1\boldsymbol\om)^{2-p}\rho^p.
\]
Now, we determine $A$ such that the set inclusion
$$
	\big( - \theta \big( \tfrac{1}{4}\rho \big)^p, 0 \big]
	\subset
	\big[ t_*+\tfrac12 b \boldsymbol\om^{q-1} (\eta_1\boldsymbol\om)^{2-p} \rho^p, t_* + b \boldsymbol \om^{q-1} (\eta_1\boldsymbol\om)^{2-p} \rho^p \big]
$$
is satisfied.
To this end, we first consider the requirement $0 \leq t_* + b \boldsymbol\om^{q-1} (\eta_1\boldsymbol\om)^{2-p} \rho^p$, which follows if the stronger condition
$$
	0 \leq
	- A \big( \tfrac14 \boldsymbol \omega \big)^{q+1-p} \rho^p + b \boldsymbol\om^{q-1} (\eta_1\boldsymbol\om)^{2-p} \rho^p
$$
is fulfilled.
This leads to the choice
$$
	A = b 4^{q+1-p} \eta_1^{2-p}.
$$
Note that we may assume $A>1$, since we could choose a smaller constant $\eta_1$ in the definition of $A$ by Remark \ref{rem:smaller eta} and use the fact that $p>2$.
The second requirement $- \theta \big( \tfrac{1}{4}\rho \big)^p \geq t_*+\tfrac12 b \boldsymbol\om^{q-1} (\eta_1\boldsymbol\om)^{2-p} \rho^p$ is satisfied if we are able to verify the stronger condition
\begin{align*}
	- (A-1) \theta \rho^p - \tfrac12 \nu \theta \rho^p
	+ \tfrac12 b \boldsymbol\om^{q-1} (\eta_1\boldsymbol\om)^{2-p} \rho^p
	&=
	- (A-1) \theta \rho^p
	- \tfrac12 \nu \theta \rho^p
	+ \tfrac12 A \theta \rho^p \\
	&\leq
	- \theta \big( \tfrac{1}{4}\rho \big)^p,
\end{align*}
which is eqivalent to
$$
	1 - \tfrac12 \nu + \tfrac{1}{4^p} \leq \tfrac12 A.
$$
Since $\nu \in (0,1)$, the last inequality holds true if $A \geq 4$.
However, as mentioned above, we may assume it by making $\eta_1$ smaller. 
Altogether, the above analysis determines $A$ through $\eta_1$ and yields a reduction of oscillation
\begin{equation}\label{Eq:reduc-osc-int-2}
\essosc_{Q_{\frac14\rho}(\theta)}u\le (1-\eta_1)\boldsymbol\om.
\end{equation}

To summarize, let us define 
\[
\eta=\min\Big\{\tfrac12\xi,\,\tfrac12\eta_o,\,\eta_1\Big\}
\in(0,\tfrac12),
\]
where
 $\frac12\xi$ is as in \eqref{Eq:reduc-osc-int-1},
$\frac12\eta_o$ is as in \eqref{Eq:reduc-osc-int-3} and $\eta_1$ is as in \eqref{Eq:reduc-osc-int-2}. 
Combining \eqref{Eq:reduc-osc-int-1} -- \eqref{Eq:reduc-osc-int-2} gives
the reduction of oscillation
 \begin{equation}\label{Eq:reduc-osc-int-4}
\essosc_{Q_{\frac14\rho}(\theta)}u\le (1-\eta)\boldsymbol\om,
 \end{equation}
provided the intrinsic relation \eqref{Eq:start-cylinder} is verified and under \eqref{Eq:Hp-main}$_1$ and \eqref{Eq:mu-pm-}.

In order to iterate the above argument, we introduce
\[
\boldsymbol\om_1=\max\Big\{(1-\eta)\boldsymbol\om, L\rho^{\frac1{p-q-1}}\Big\};
\]
we need to choose $\rho_1=\lm\rho$ for some $\lm\in(0,1)$, such that
\[
	Q_{\rho_1}(A\theta_1)\subset Q_{\frac14\rho}(\theta) \cap Q_o,
	\quad\text{ where } \theta_1= \big( \tfrac14\boldsymbol\om_1 \big)^{q+1-p}.
\]
To this end, we first let 
\[
\lm=\tfrac14 A^{-\frac1p}(1-\eta)^{\frac{p-q-1}p}.
\]
and estimate
\[
	A\theta_1\rho_1^p
	=
	A\big(\tfrac14\boldsymbol\om_1\big)^{q+1-p} (\lambda\rho)^p
	\le 
	\big(\tfrac14\boldsymbol\om\big)^{q+1-p} \big(\tfrac14\rho\big)^p
	=
	\theta \big(\tfrac14\rho\big)^p.
\]
consequently, the first set inclusion $Q_{\rho_1}(A\theta_1)\subset Q_{\frac14\rho}(\theta)$ holds.
Note that $\lambda<\frac14$.
The second set inclusion $Q_{\rho_1}(A\theta_1)\subset Q_o$ is verified similarly with the same choice of $\lm$.
Therefore, taking into account \eqref{Eq:extra-control}, \eqref{Eq:reduc-osc-int-4}
and the violation of \eqref{Eq:mu-pm-}, i.e.~the case where $\essosc_{Q_o} u = \boldsymbol \mu^+ - \boldsymbol \mu^- \leq \frac{1}{2} \boldsymbol \omega$, we arrive at the intrinsic relation
\[
\essosc_{Q_{\rho_1}(A\theta_1)}u\le \boldsymbol\om_1,
\]
which takes the place of \eqref{Eq:start-cylinder} in the next stage. 

\subsection{Reduction of Oscillation Near Zero Concluded}\label{S:case-1-3}
Now we may proceed by induction. 
Suppose that, up to $i=1,2,\cdots j-1$, we have built
\begin{equation*}
\left\{
	\begin{array}{c}
	\rho_o=\rho,
	\quad
	\dsty\varrho_i=\lm\varrho_{i-1},
	\quad
	\theta_i=\big(\tfrac14\boldsymbol\om_i\big)^{q+1-p}\\[5pt]
	\boldsymbol\om_o=\boldsymbol\om,
	\quad
	\boldsymbol\om_i=\max\Big\{(1-\eta)\boldsymbol\om_{i-1},L\rho_{i-1}^{\frac1{p-q-1}}\Big\},\\[5pt]
	Q_i=Q_{\rho_i}(\theta_i),
	\quad
	Q'_i=Q_{\frac14\rho_i}(\theta_i)\\[5pt]
	\dsty\boldsymbol\mu_i^+=\essup_{Q_i}u,
	\quad
	\boldsymbol\mu_i^-=\essinf_{Q_i}u,
	\quad
	\essosc_{Q_i}u\le\boldsymbol\om_i.
	\end{array}
\right.
\end{equation*}
For all the indices $i=1,2,\cdots j-1$, we alway assume that  \eqref{Eq:Hp-main}$_1$ holds true, i.e.,
$$
	\boldsymbol\mu_i^-\le\xi\boldsymbol\om_i\quad
	\text{ and }\quad\boldsymbol\mu_i^+\ge-\xi\boldsymbol\om_i.
$$
By this means the previous arguments   can be repeated and we have for all $i=1,2,\cdots j$,
\[
Q_{\rho_i}(A\theta_i)\subset Q'_{i-1}, \quad	\essosc_{Q_i}u\le(1-\eta)\boldsymbol\om_{i-1}\le\boldsymbol\om_i.
\]
Consequently, iterating the above recursive inequality we obtain for all $i=1,2,\cdots j$,
\begin{equation}\label{Eq:case1}
	\essosc_{Q_i}u
	\le
	\boldsymbol\om_i
	\le
	\max \Big\{ (1-\eta)^i \boldsymbol \omega, L \rho^\frac{1}{p-q-1} \Big\}
	=
	\max \Big\{ \boldsymbol\om \Big(\frac{\rho_i}{\rho}\Big)^{\be_o},
	L\rho^\frac1{p-q-1} \Big\},
\end{equation}
where
$$
	\beta_o=\frac{\ln(1-\eta)}{\ln \lm}.
$$
\subsection{Reduction of Oscillation Away From Zero}\label{S:case-2}
In this section, let us suppose $j$ is the first index satisfying the second case in \eqref{Eq:Hp-main}, i.e.
\[
\text{either \quad$\boldsymbol\mu_j^-> \xi\boldsymbol \om_j$ \quad or\quad $\boldsymbol\mu_j^+< -\xi\boldsymbol \om_j$.}
\]
Let us treat for instance $\boldsymbol\mu_j^->\xi\boldsymbol\om_j$,
for the other case is analogous.
We observe that since $j$ is the first index for this to happen,
one should have $\boldsymbol\mu_{j-1}^+ \le \boldsymbol\mu_{j-1}^- + \boldsymbol \omega_{j-1} \leq (1+\xi) \boldsymbol \om_{j-1}$.
Here, we assume that there exists an index $j-1$ such that the first case in \eqref{Eq:Hp-main} is fulfilled.
This can be justified by choosing  $\boldsymbol \omega =  \frac{1}{\xi}   \| u \|_{L^\infty(E_T)}$ in Section \ref{S:5:1}.
Moreover, since $Q_j \subset Q_{j-1}$, by the definition of the essential supremum one estimates
\[
	\boldsymbol\mu_{j}^-
	\leq
	\boldsymbol\mu_j^+
	\leq
	\boldsymbol\mu_{j-1}^+
	\leq
	(1+\xi) \boldsymbol \om_{j-1}
	\leq
	\frac{1+\xi}{1-\eta}\boldsymbol \om_{j}.
\]
As a result, we have
\begin{equation}\label{Eq:6:9}
\xi\boldsymbol \om_{j}\le\boldsymbol \mu_{j}^-\le\frac{1+\xi}{1-\eta}\boldsymbol \om_{j}.
\end{equation}
The bound \eqref{Eq:6:9} indicates that starting from $j$ the equation \eqref{Eq:1:1}
resembles the parabolic $p$-Laplacian type equation in $Q_j$. 
We drop the suffix $j$ from our notation for simplicity,
and introduce $v:=u/\boldsymbol\mu^-$ in $Q=K_{\rho}\times(-\theta\rho^p,0]$, where $\theta = \big( \frac{1}{4} \boldsymbol \omega )^{q+1-p}$.
It is straightforward to verify that $v$ belongs to the function space \eqref{Eq:1:3p} defined on $Q$ and satisfies
\begin{equation*}
	\pl_tv^{q}-\dvg\bar{\bl{A}}(x,t,v, Dv)=0\quad\text{ weakly in }Q,
\end{equation*}
where, for $ (x,t)\in Q$, $v\in\rr$ and $\z\in\rn$, we have defined 
$$
	\bar{\bl{A}}(x,t,v, \z)
	=
	\bl{A}(x,t,\boldsymbol\mu^-v, \boldsymbol\mu^-\z)/(\boldsymbol\mu^-)^{q},
$$
which is subject to the structure conditions
\begin{equation*}
	\left\{
	\begin{array}{c}
		\bar{\bl{A}}(x,t,v,\z)\cdot \z\ge C_o (\boldsymbol \mu^-)^{p-q-1} |\z|^p \\[5pt]
		|\bar{\bl{A}}(x,t,v,\z)|\le C_1 (\boldsymbol \mu^-)^{p-q-1} |\z|^{p-1}
	\end{array}
	\right .
	\qquad \mbox{ for a.e.~$(x,t)\in Q$, $\forall\, v\in\rr$, $\forall\, \z\in\rn$.}
\end{equation*}
Moreover, since $\boldsymbol \omega/\boldsymbol \mu^- \leq 1/\xi$, we have that
\begin{equation}\label{Eq:6:10}
1\le v\le\frac{\boldsymbol\mu^+}{\boldsymbol \mu^-}\le\frac{\boldsymbol \mu^-+\boldsymbol\om}{\boldsymbol \mu^-}\le
\frac{1+\xi}{\xi}\quad\text{ a.e. in }Q.
\end{equation}
To proceed, it turns out to be more convenient to consider 
$w:=v^{q}$,
 which because of \eqref{Eq:6:10} belongs to the function space \eqref{Eq:1:3p}$_{q=1}$ defined on $Q$ and satisfies
\begin{equation*}
\partial_tw-\dvg\widetilde{\bl{A}}(x,t,w, Dw)=0\quad\text{ weakly in }Q,
\end{equation*}
where  we have defined  the vector-field $\widetilde{\bf A}$ by
\[
	\widetilde{\bl{A}}(x,t,y, \zeta)
	=\
	\bar{\bl{A}}\Big( x,t,\widetilde y^{\frac{1}{q}},\tfrac{1}{q} \widetilde y^{\frac{1-q}{q}}
	\zeta\Big),
\]
for a.e.~$(x,t)\in Q$, 
any $y\in\rr$ and any $\zeta\in\rn$. This time $\widetilde y$
is defined by
\[
	\widetilde y :=\min\Big\{ \max\big\{y,\tfrac12\big\}, 2\Big(\frac{1+\xi}{\xi}\Big)^q\Big\}.
\]

Employing \eqref{Eq:6:10} again, we verify  that there exist
positive constants $\widetilde{C}_o=\boldsymbol\gamma_o (p,q,\xi)C_o$ and $\widetilde{C}_1= \boldsymbol\gamma_1 (p,q,\xi)C_1$, such that
\begin{equation*}
	\left\{
	\begin{array}{c}
		\widetilde{\bl{A}}(x,t,y,\zeta)\cdot \zeta
		\geq
		\widetilde{C}_o (\boldsymbol \mu^-)^{p-q-1} |\zeta|^p \\[5pt]
		|\widetilde{\bl{A}}(x,t,y,\zeta)|
		\leq
		\widetilde{C}_1 (\boldsymbol \mu^-)^{p-q-1} |\zeta|^{p-1},
	\end{array}
	\right.
	\qquad \mbox{ for a.e.~$(x,t)\in Q$, $\forall\, y\in\rr$, $\forall\, \z\in\rn$.}
\end{equation*}
Note that $\xi$ is already fixed in \eqref{Eq:choice-xi} in terms of the data.
To proceed, we introduce the function
$$
	\widehat{w}(x,t)
	:=
	w(x, (\boldsymbol \mu^-)^{q+1-p} t),
$$
which satisfies
\begin{equation}
	\label{Eq:6:11}
	\partial_t \widehat{w} - \dvg\widehat{\bl{A}}(x,t, \widehat{w}, D \widehat{w}) = 0
	\quad\text{ weakly in }
	\widehat{Q}
	:=
	K_\rho \times \big(-(\boldsymbol \mu^-)^{p-q-1} \theta \varrho^p, 0\big]
\end{equation}
and  belongs to the function space \eqref{Eq:1:3p}$_{q=1}$ defined on $\widehat{Q}$.
Here the function $\widehat{\bl{A}}$ is defined by
$$
	\widehat{\bl{A}}(x,t, y, \zeta)
	:=
	(\boldsymbol \mu^-)^{q+1-p} \widetilde{\bl{A}}(x, (\boldsymbol \mu^-)^{q+1-p} t, y, \zeta)
$$
and subject to the structure conditions
\begin{equation}
\label{Eq:6:12}
	\left\{
	\begin{array}{c}
		\widehat{\bl{A}}(x,t,y,\zeta)\cdot \zeta
		\geq
		\widetilde{C}_o |\zeta|^p, \\[5pt]
		|\widehat{\bl{A}}(x,t,y,\zeta)|
		\leq
		\widetilde{C}_1 |\zeta|^{p-1},
	\end{array}
	\right.
	\qquad \mbox{ for a.e.~$(x,t)\in \widehat{Q}$, $\forall\, y\in\rr$, $\forall\, \z\in\rn$.}
\end{equation}
This shows that $\widehat{w}$ is a local weak solution to the parabolic $p$-Laplacian type equation in $\widehat{Q}$.

First proved in \cite{DB86} the power-like oscillation decay for solutions to this kind of
degenerate  parabolic equation is well known by now.  
We state the conclusion in the following proposition in a form that favors our application,
and refer to the monographs \cite{DB, Urbano-08} for a comprehensive treatment of this issue. 
\begin{proposition}\label{Prop:5:1}
Let $p>2$, $\sig$ in $(0,1)$ and $\boldsymbol{\widehat{\om}}>0$. Then, there exist constants $\beta_1$ in $(0,1)$ and $\boldsymbol\gm>1$ depending only on the data
$N,p,\widetilde C_o, \widetilde C_1$ and $\sig$, such that there holds: Whenever $\widehat{w}$ is a bounded, local, weak solution to 
\eqref{Eq:6:11} -- \eqref{Eq:6:12} in $\widehat{Q}$, such that with $\widehat\theta = \boldsymbol{\widehat{\om}}^{2-p}$ 
the assumptions
\begin{equation}\label{Eq:5:8}
	\essosc_{Q_{\sig\varrho}(\widehat\theta)} \widehat{w}
	\le \boldsymbol{\widehat{\om}}
	\quad
	\text{ and }
	\quad
	Q_{\sig\varrho}(\widehat\theta) \subset \widehat{Q},
\end{equation}
hold true, then for all $0<r\le\rho$ we have
\[
	\essosc_{Q_r(\widehat\theta)} \widehat{w}
	\leq
	\boldsymbol \gm \boldsymbol{\widehat{\om}} \Big(\frac{r}{\rho}\Big)^{\be_1}.
\]
\end{proposition}

We tend to use Proposition~\ref{Prop:5:1}.
First we check the condition \eqref{Eq:5:8} is satisfied.
Indeed, by  the mean value theorem  
 and \eqref{Eq:6:10} there exists some positive $\widetilde{\boldsymbol\gamma}=\widetilde{\boldsymbol\gamma}(q,\xi)$, such that
$$
	\essosc_{\widehat{Q}} \widehat{w}
	=
	\essosc_Q w
	\leq
	\widetilde{\boldsymbol\gamma}  \essosc_Q v
	\leq
	\widetilde{\boldsymbol\gamma}  \frac{\boldsymbol \omega}{\boldsymbol \mu^-}
	=:
	\widehat{\boldsymbol \omega}.
$$
According to \eqref{Eq:6:9} we find that
$$
	\frac{1-\eta}{1+\xi}
	\leq
	\frac{\boldsymbol \omega}{\boldsymbol \mu^-}
	\leq
	\frac{1}{\xi}.
$$
Further, by definition of the corresponding cylinders, we obtain that $Q_{\sig\varrho}(\widehat\theta) \subset \widehat{Q}$, provided
$$
	\Big( \widetilde{\boldsymbol\gamma} \frac{\boldsymbol \omega}{\boldsymbol \mu^-} \Big)^{2-p}
	(\sigma \rho)^p
	\leq
	( \boldsymbol \mu^- )^{p-q-1} 
	\big( \tfrac{1}{4} \boldsymbol \omega \big)^{q+1-p} \rho^p
$$
holds true.
This can be achieved by choosing $\sigma$ small enough, i.e.
$$
	\sigma \leq
	 \big( \tfrac{1}{4} \big)^{q+1-p} \widetilde{\boldsymbol\gamma}^{p-2} \Big( \frac{\boldsymbol \omega}{\boldsymbol \mu^-} \Big)^{q-1}.
$$
In view of the lower and upper bound on the ratio $\boldsymbol \omega/\boldsymbol \mu^-$, 
the number $\sigma$ can be chosen only in terms of the data, such that
$$
	\essosc_{Q_{\sig\varrho}(\widehat\theta)} \widehat{w}
	\leq
	\widehat{\boldsymbol \omega},
$$
i.e.~the condition \eqref{Eq:5:8} is fulfilled.
Consequently, by Proposition \ref{Prop:5:1} we have
$$
	\essosc_{Q_r(\widehat\theta)} \widehat{w}
	\leq
	\boldsymbol \gamma \widehat{\boldsymbol \omega} \Big(\frac{r}{\rho}\Big)^{\be_1}
	\leq
	\bar{\boldsymbol \gamma} \Big(\frac{r}{\rho}\Big)^{\be_1}
$$
for $\bar{\boldsymbol \gamma}=\boldsymbol \gamma\widetilde{\boldsymbol \gamma}/\xi$ and for any $0 < r \le \rho$, with some $\beta_1\in (0,1)$ depending only on the data.
Since $p>2$, we may estimate
$$
\widehat\theta > \widehat{\theta}_o := \Big(\frac{\widetilde{\boldsymbol \gamma}}{\xi } \Big)^{2-p}
$$ 
and conclude that
$$
	\essosc_{Q_r(\widehat{\theta}_o)} \widehat{w}
	\leq
	\bar{\boldsymbol\gamma} \Big(\frac{r}{\rho}\Big)^{\be_1} \qquad\forall\, 0<r\le\rho.
$$
Reverting to $w$ and using the fact that $q+1 < p$ and \eqref{Eq:6:9} in order to estimate 
$$(\boldsymbol \mu^-)^{q+1-p} \geq \Big( \frac{1 + \xi}{1 - \eta} \boldsymbol \omega \Big)^{q+1-p},$$ 
we obtain that
$$
	\essosc_{Q_r \left( \widehat{\theta}_1 \boldsymbol \omega^{q+1-p} \right)} w
	\leq
	\bar{\boldsymbol\gamma} \Big(\frac{r}{\rho}\Big)^{\be_1},
$$
where 
$$
\widehat{\theta}_1 := \Big( \frac{1 + \xi}{1 - \eta} \Big)^{q+1-p} \widehat{\theta}_o
$$ 
depends only on the data.
Recalling the definition of $w$, by   the mean value theorem 
and \eqref{Eq:6:10} one easily estimates that  for some positive $\widetilde{\boldsymbol\gamma}=\widetilde{\boldsymbol\gamma}(q,\xi)$,
$$
	\essosc_{Q_r ( \widehat{\theta}_1 \boldsymbol \omega^{q+1-p} )} v
	\leq
	\widetilde{\boldsymbol\gamma}(q,\xi)\essosc_{Q_r ( \widehat{\theta}_1 \boldsymbol \omega^{q+1-p} )} w.
$$
Finally, we revert to $u$ and the suffix $j$, and use \eqref{Eq:6:9} to estimate $\boldsymbol \mu^-_j \leq \frac{1 + \xi}{1 - \eta} \boldsymbol \omega_j$, which leads to
\begin{equation}\label{Eq:case2}
    \essosc_{Q_r ( \widehat{\theta}_1 \boldsymbol \omega^{q+1-p}_j )} u
    \le
	\boldsymbol \mu^-_j \essosc_{Q_r ( \widehat{\theta}_1 \boldsymbol \omega^{q+1-p}_j )} v
	\leq
	\boldsymbol\gamma \boldsymbol \omega_j \Big(\frac{r}{\rho_j}\Big)^{\be_1},
\end{equation}
whenever $0<r<\rho_j$.
Since $\rho \leq 1$, we have that $\boldsymbol \omega_j \leq \boldsymbol \omega_1 \leq \max\{ \boldsymbol \omega, L \} =: \boldsymbol \omega_L$ and therefore we obtain that $Q_r ( \widehat{\theta}_1 \boldsymbol \omega_L^{q+1-p} ) \subset Q_r ( \widehat{\theta}_1 \boldsymbol \omega^{q+1-p}_j )$.
Combining this with \eqref{Eq:case1} and \eqref{Eq:case2},
we arrive at the following: for all $0<r<\rho$,
\[
	\essosc_{Q_r \left( \widehat{\theta}_1 \boldsymbol \omega_L^{q+1-p} \right)} u
	\leq
	\boldsymbol{\gm}  \boldsymbol \om  \Big(\frac{r}{\rho}\Big)^{\be_2}+\boldsymbol{\gm} L\rho^\frac{1}{p-q-1},
	\qquad\text{where } \be_2=\min\{\beta_o,\beta_1\}.
\]
Without loss of generality, we may assume the above oscillation estimate holds with $\rho$ replaced by some $\widetilde{\rho}\in(r,\rho)$.
Then taking $\widetilde{\rho}=(r\rho)^\frac12$ and properly adjusting the H\"older exponent, we obtain the power-like decay of oscillation
$$
	\essosc_{Q_r \left( \widehat{\theta}_1 \boldsymbol \omega_L^{q+1-p} \right)} u
	\leq
	\boldsymbol \gamma \boldsymbol \omega \Big( \frac{r}{\rho} \Big)^\frac{\beta_2}{2} + \boldsymbol \gamma L \rho^\frac{1}{p-q-1} \Big( \frac{r}{\rho} \Big)^\frac{1}{2(p-q-1)}
	\leq
	\boldsymbol \gamma \boldsymbol \omega_L 
	\Big( \frac{r}{\rho} \Big)^\beta,
$$
where
$$
	\beta = \min\Big\{ \frac{\beta_2}{2}, \frac{1}{2(p-q-1)} \Big\}.
$$
At this stage, the proof of Theorem \ref{Thm:1:1} can be completed by  a standard covering argument.
\section{The Second Proof of Theorem \ref{Thm:1:1}}\label{S:6}
The purpose of this section is to present another proof of Theorem \ref{Thm:1:1}
without using the expansion of positivity (Lemma~\ref{Lm:expansion:p}).
As we shall see, the arguments in Section~\ref{S:6:2} are similar to that of Section~\ref{S:case-1}.
The main difference appears in Section~\ref{S:6:3}. 
To avoid using Lemma~\ref{Lm:expansion:p} as done in Section~\ref{S:case-1-2},
we perform an argument of DiBenedetto \cite{DB86}, adapted in \cite{BDL}.
The virtual advantage of this section is that the proof relies solely on 
the energy estimates in Proposition~\ref{Prop:2:1}. As such it offers an amenable adaption
near the boundary given Neumann data, cf.~Section~\ref{S:Neumann}.

\subsection{The Proof Begins}
The set-up is the same as in Section~\ref{S:5:1}.
Namely, we introduce the quantities 
$\{\boldsymbol \mu^\pm,\,\boldsymbol \om ,\, \theta, \, L,\, A\}$ and the cylinders $Q_\rho(A\theta)\subset Q_o$.
Moreover, they are connected by the intrinsic relation \eqref{Eq:start-cylinder}.
For a positive $\xi$ to be determined, the proof unfolds along two main cases, as in \eqref{Eq:Hp-main}.

\subsection{Reduction of Oscillation Near Zero--Part I}\label{S:6:2}
Like in Section~\ref{S:case-1}, we assume that \eqref{Eq:Hp-main}$_1$ holds and work with $u$ as a super-solution near its infimum.
Then we proceed with the assumption \eqref{Eq:mu-pm-}, which implies one of \eqref{Eq:mu-pm} holds. 
We may take \eqref{Eq:mu-pm}$_1$, such that $\frac14 \boldsymbol\om \le \boldsymbol\mu^+ \le 2\boldsymbol\om$.

The second proof departs from here.
Suppose that for some $\bar{t}\in\big(-(A-1)\theta\rho^p,0\big]$,
\begin{equation}\label{Eq:4:2}
	\Big|\Big\{u\le\boldsymbol\mu^-+\tfrac14 \boldsymbol\om\Big\}
	\cap 
	(0,\bar{t})+Q_{\varrho}(\theta)\Big|\le \nu|Q_{\varrho}(\theta)|,
\end{equation}
where $\nu$ is the constant determined in Lemma~\ref{Lm:DG:1} in terms of the data.
According to Lemma~\ref{Lm:DG:1} applied with $\xi=\frac14$, we have
\[
	u\ge\boldsymbol \mu^-+\tfrac{1}8\boldsymbol\om
	\quad
	\mbox{a.e.~in $(0,\bar{t})+Q_{\frac12 \varrho}(\theta)$,}
\]
since the other alternative, i.e., $|\boldsymbol\mu^-|\ge 2\boldsymbol \om$,
does not hold due to \eqref{Eq:Hp-main}$_1$.
This pointwise information parallels \eqref{Eq:initial-condi} in Section~\ref{S:case-1}.
Similar arguments can be reproduced as in Section~\ref{S:case-1} to obtain 
the reduction of oscillation as in \eqref{Eq:reduc-osc-int-1} -- \eqref{Eq:reduc-osc-int-3}.
In particular, only Lemma~\ref{Lm:DG:1}, Lemma~\ref{Lm:DG:initial:1} and Lemma~\ref{Lm:DG:initial:2} are used.
In this process we fix the constant $\xi$ as in \eqref{Eq:choice-xi} depending on the data
and $A$, which will be chosen next in terms of the data.
\subsection{Reduction of Oscillation Near Zero--Part II}\label{S:6:3}
In this section we still assume that \eqref{Eq:Hp-main}$_1$ holds.
However, now we work with $u$ as a sub-solution near its
supremum.
Keep also in mind that \eqref{Eq:mu-pm}$_1$ is enforced,
such that $\frac14 \boldsymbol\om\le\boldsymbol\mu^+\le 2\boldsymbol\om$ may be assumed.

Suppose contrary to \eqref{Eq:4:2} that, recalling $\theta=(\frac14\boldsymbol\om)^{q+1-p}$, 
\begin{equation*}
	\Big|\Big\{u\le\boldsymbol\mu^-+\tfrac14 \boldsymbol\om\Big\}
	\cap 
	(0,\bar{t})+Q_{\varrho}(\theta)\Big|> \nu|Q_{\varrho}(\theta)|,\qquad\forall\, \bar{t}\in\big(-(A-1)\theta\rho^p,0\big].
\end{equation*}
Then for any such $\bar{t}$, it is easy to see that there exists some 
$s\in\big[\bar{t}-\theta\varrho^p,\bar{t}-\tfrac{1}2\nu\theta\varrho^p\big]$ with
\begin{equation*}
	\Big|\Big\{u(\cdot, s)\le\boldsymbol \mu^-+\tfrac{1}4\boldsymbol\om\Big\}\cap K_{\varrho}\Big|>\tfrac{1}2\nu |K_{\varrho}|.
\end{equation*}
Since we assumed that $\boldsymbol \mu^+ - \boldsymbol \mu^- > \frac12 \boldsymbol \omega$, there holds $\boldsymbol\mu^+-\frac14\boldsymbol\om>\boldsymbol\mu^-+\frac14\boldsymbol\om$, which implies
\begin{equation*}
	\Big|\Big\{u(\cdot,s)\le\boldsymbol \mu^+-\tfrac{1}4\boldsymbol\om\Big\}\cap K_\varrho\Big|
	\ge\tfrac12\nu|K_{\varrho}|.
\end{equation*}
Recall that due to \eqref{Eq:mu-pm}$_1$ we have
$\frac14 \boldsymbol\om\le\boldsymbol\mu^+\le 2\boldsymbol\om$. 
Thus our assumptions  for the following Sections~\ref{S:4:3:1} -- \ref{S:4:3:3} are
\begin{equation}\label{Eq:4:3a}
	\tfrac14 \boldsymbol\om\le\boldsymbol\mu^+\le2\boldsymbol\om,
\end{equation}
and
\begin{equation}\label{Eq:4:3b}
	\left\{
	\begin{array}{c}
	\mbox{for any $\bar{t}\in\big(-(A-1)\theta\varrho^p,0\big]$ there exists $s\in\big[\bar{t}-\theta\varrho^p, \bar{t}
	-\tfrac12 \nu\theta\varrho^p
	\big]$}\\[6pt]
	\mbox{such that   	$\dsty\Big|\Big\{u(\cdot,s)\le\boldsymbol \mu^+-\tfrac{1}4\boldsymbol\om\Big\}\cap K_\varrho\Big|
	\ge\tfrac12\nu|K_{\varrho}|. $  }
	\end{array}
	\right.
\end{equation}
They would allow us to determine $A$ and reduce the oscillation in this case.
Similar arguments in Sections~\ref{S:4:3:1} -- \ref{S:4:3:3} have been carried out in \cite{BDL}.
However we think it is necessary to adapt them in the new setting because of the technical nature.
\subsubsection{Propagation of Measure Theoretical Information}\label{S:4:3:1}
\begin{lemma}\label{Lm:4:2}
Suppose \eqref{Eq:4:3a} and  \eqref{Eq:4:3b} are in force.
There exists $\eps\in(0,1)$,
depending only on $\nu$  and the data, such that
\begin{equation*}
	\Big|\Big\{ u(\cdot, t)\le\boldsymbol \mu^+-\eps\boldsymbol \om\Big\}\cap K_{\varrho}\Big|
	\ge
	\tfrac14\nu |K_\varrho|
	\quad\mbox{for all $t\in(s,\bar{t}\,]$.}
\end{equation*}
\end{lemma}
\begin{proof}
For ease of notation, we set $s=0$.
Further, for $\dl>0$ and $0<\varep\le\frac18$ to be determined by the data and $\nu$, we consider $Q:=K_{\varrho}\times(0,\dl\varep^{2-p}\theta\varrho^p]$ and $k=\boldsymbol\mu^+-\varep\boldsymbol \om \ge\frac18\boldsymbol\om$.
Applying the energy estimate in Proposition~\ref{Prop:2:1} with a standard non-negative time independent cutoff function 
$\zeta (x,t)\equiv\z(x)$ that equals $1$ on $K_{(1-\sig)\varrho}$ for some $\sigma\in(0,1)$ to be fixed later, vanishes on $\pl K_{\varrho}$ and satisfies
$|D\z|\le(\sig\varrho)^{-1}$, we obtain for all $0<t<\dl \varep^{2-p}\theta\rho^p$ that
\begin{align*}
	\int_{K_\varrho\times\{t\}}&\int_{k}^u \tau^{q-1}(\tau-k)_+\,\d\tau \, \z^p\,\dx\\
	&\le
	\int_{K_\varrho\times\{0\}}\int_{k}^u \tau^{q-1}(\tau-k)_+\,\d\tau \, \z^p\,\dx
	+
	\boldsymbol\gm\iint_{Q}(u-k)^{p}_+|D\z|^p\,\dx\dt.
\end{align*}
Defining $k_{\tilde\eps}=\boldsymbol \mu^+-\tilde\eps\varep \boldsymbol \om $ for some $\tilde\eps\in(0,\frac12)$, we estimate the term on the left-hand side by
\begin{align*}
		\int_{K_\varrho\times\{t\}}&\int_{k}^u \tau^{q-1}(\tau-k)_+\,\d\tau \, \z^p\,\dx
	\ge\big|\big\{ u(\cdot, t)>k_{\tilde\eps}\big\}\cap K_{(1-\sig)\varrho}\big|
	\int^{k_{\tilde\eps}}_k \tau^{q-1}(\tau-k)_+\,\d\tau.
\end{align*}
Further, note that by the mean value theorem and the restriction $\frac14 \boldsymbol\om\le\boldsymbol\mu^+\le2\boldsymbol\om$, there exists a constant $\boldsymbol \gamma = \boldsymbol \gamma(q)$ such that
\begin{equation*}
	\int^{k_{\tilde\eps}}_k \tau^{q-1}(\tau-k)_+\,\d\tau
	\ge
	\boldsymbol\gm\boldsymbol \om^{q-1}(\varep\boldsymbol\om)^2
	=
	\boldsymbol\gm\varep^2\boldsymbol \om^{q+1}.
\end{equation*}
Next, by \eqref{Eq:4:3b} we obtain for the first term on the right-hand side of the energy estimate that
\begin{align*}
	\int_{K_\varrho\times\{0\}}\int_{k}^u \tau^{q-1}(\tau-k)_+\,\d\tau \, \z^p\,\dx
	\le
	\big(1-\tfrac12\nu\big)|K_{\rho}|
	\int^{\boldsymbol\mu^+}_k \tau^{q-1}(\tau-k)_+\,\d\tau
\end{align*}
and by the choice of $\zeta$ and $u \leq \boldsymbol\mu^+$ for the second term on the right-hand side that
\begin{equation*}
	\iint_{Q}(u-k)^{p}_+|D\z|^p\,\dx\dt
	\le
	\frac{\boldsymbol \gm\dl}{\sig^p}\varep^{2-p}\theta(\varep\boldsymbol\om)^p|K_\varrho|
	\le
	\frac{\boldsymbol \gm\dl}{\sig^p}\varep^{2}\boldsymbol\om^{q+1}|K_\varrho|.
\end{equation*}
Combining the preceding estimates leads to
\begin{align*}
	\big|\big\{ u(\cdot, t)>k_{\tilde\eps}\big\}\cap K_{(1-\sig)\varrho}\big|
	\le 
	\frac{\dsty\int^{\boldsymbol\mu^+}_k \tau^{q-1}(\tau-k)_+\,\d\tau}{\dsty\int^{k_{\tilde\eps}}_k \tau^{q-1}(\tau-k)_+\,\d\tau}
	\big(1-\tfrac12\nu\big)|K_{\rho}| +\frac{\boldsymbol\gm\dl}{\sig^p}|K_\rho|.
\end{align*}
Rewriting the fractional number of integrals on the right-hand side
and using the mean value theorem as well as the restrictions $\frac14 \boldsymbol\om\le\boldsymbol\mu^+\le2\boldsymbol\om$ and $k\ge\frac18 \boldsymbol\om$ yields the bound
\begin{align*}
	\frac{\dsty\int^{\boldsymbol\mu^+}_k \tau^{q-1}(\tau-k)_+\,\d\tau}{\dsty\int^{k_{\tilde\eps}}_k \tau^{q-1}(\tau-k)_+\,\d\tau}
	=
	1+
	\frac{\dsty\int^{\boldsymbol\mu^+}_{k_{\tilde\eps}} \tau^{q-1}(\tau-k)_+\,\d\tau}
	{\dsty\int^{k_{\tilde\eps}}_{k} \tau^{q-1}(\tau-k)_+\,\d\tau}
	\leq
	1 + \boldsymbol\gm\tilde\eps,
\end{align*}
where $\boldsymbol\gm$ depends only on $q$.
Inserting this into the previous inequality, we conclude that
\begin{align*}
	\big|\big\{ u(\cdot, t)>k_{\tilde\eps}\big\}\cap K_{\varrho}\big|
	\le 
	\big(1-\tfrac12\nu\big)\big(1+\boldsymbol\gm\tilde\eps\big)
	|K_{\rho}|
	+
	\frac{\boldsymbol\gm\dl}{\sig^p}|K_\varrho|
	+
	N\sig |K_\varrho|.
\end{align*}
Now, we first fix $\tilde\eps = \tilde\eps(q, \nu)$ small enough that
\begin{equation*}
	\big(1-\tfrac12\nu\big)\big(1+\boldsymbol\gm\tilde\eps\big)
	\le 
	1-\tfrac38\nu
\end{equation*}
and define $\sig:=\frac{\nu}{16N}$.
Then, we choose $\delta$ small enough that $\frac{\boldsymbol\gm\dl}{\sig^p}\le\tfrac1{16} \nu$ and $\varep$ small enough that $\dl\varep^{2-p}\ge1$, where we take into account that $p>2$.
Redefining $\tilde\eps\varep$ as $\eps$, we finish the proof of the lemma.
\end{proof}
Since $\bar{t}$ is arbitrary in $( -(A-1) \theta \rho^p, 0 ]$, the previous lemma actually yields the measure theoretical information
\begin{equation}\label{Eq:4:4}
	\Big|\Big\{ u(\cdot, t)\le\boldsymbol\mu^+-\eps\boldsymbol\om\Big\}
	\cap K_{\varrho}\Big|\ge\tfrac{1}4 \nu |K_\varrho|
\quad\mbox{ for all $t\in\big(-(A-1)\theta\varrho^p,0\big]$.}
\end{equation}
\subsubsection{Shrinking the Measure Near the Supremum}\label{S:4:3:2}
Let $\epsilon\in (0,1)$ denote the constant from Lemma~\ref{Lm:4:2} depending only on the data.
Further, we choose the number $A$ in the form
$$
	A=2^{j_*(p-2)}+1
$$
with some $j_*$ to be fixed later and consider the cylinder $Q_\varrho((A-1)\theta)=Q_\varrho(2^{j_*(p-2)}\theta)$, where $\theta=(\frac14\boldsymbol\om)^{q+1-p}$. 
\begin{lemma}\label{Lm:A:4}
Suppose \eqref{Eq:4:3a} 
and \eqref{Eq:4:4} hold.
Then, there exists a constant $\boldsymbol \gm>0$ depending only on the data, such
that for any positive integer $j_*$, we have
\begin{equation*}
	\Big|\Big\{ u\ge\boldsymbol \mu^+-\frac{\eps\boldsymbol\om}{2^{j_*}}\Big\}\cap Q_\varrho((A-1)\theta)\Big|\le
	\frac{\boldsymbol\gm}{j_*^{\frac{p-1}p}}|Q_\varrho((A-1)\theta)|.
\end{equation*}
\end{lemma}
\begin{proof}
Consider the cylinder $K_{2\varrho}\times (-(A-1)\theta\varrho^p,0]$ and a time independent cutoff function $\z(x,t)\equiv\zeta (x)$ vanishing on $\pl K_{2\varrho}$ and equal to $1$ in $K_{\varrho}$ such that $|D\z|\le2\varrho^{-1}$.
Applying the energy estimate from Proposition~\ref{Prop:2:1} with levels $k_j=\boldsymbol \mu^+-2^{-j-1}\eps\boldsymbol\om$ for $j=0,\cdots, j_\ast-1$, we obtain that
\begin{align*}
	&\iint_{Q_\varrho((A-1)\theta)}|D(u-k_j)_+|^p\,\dx\dt\\
	&
	\le
	\int_{K_{2\varrho}\times\{ -(A-1)\theta\varrho^p\}} \z^p \mathfrak g_+ (u,k_j) \,\dx
	+\boldsymbol\gm\iint_{K_{2\varrho}\times (-(A-1)\theta\varrho^p,0]}(u-k_j)_+^p|D\z|^p\,\dx\dt.
\end{align*}
By the mean value theorem, the restriction $\frac14 \boldsymbol\om\le \boldsymbol \mu^+\le2\boldsymbol\om$ and the fact that the parameter $\eps$ is already fixed in Lemma \ref{Lm:4:2} in dependence on the data, the first term on the right-hand side of the preceding inequality is estimated by
\begin{align*}
	\int_{K_{2\varrho}\times\{ -(A-1)\theta\varrho^p\}}\z^p \mathfrak g_+ (u,k_j)\,\dx
	&\le
	\boldsymbol\gm\boldsymbol\om^{q-1}\left(\frac{\eps\boldsymbol\om}{2^j}\right)^2|K_{2\varrho}|\\
	&\le
	\frac{\boldsymbol \gm}{\varrho^p\eps^{p-2}}\left(\frac{\eps\boldsymbol\om}{2^j}\right)^p|Q_\rho((A-1)\theta)|\\
	&\le
	\frac{\boldsymbol \gm}{\varrho^p}\left(\frac{\eps\boldsymbol\om}{2^j}\right)^p|Q_\rho((A-1)\theta)|.
\end{align*}
For the second term on the right, we use $u \leq \boldsymbol \mu^+$ and the bound for $|D\zeta|$.
Thus, we arrive at
\begin{equation*}
	\iint_{Q_\varrho((A-1)\theta)}|D(u-k_j)_+|^p\,\dx\dt
	\le
	\frac{\boldsymbol\gm}{\varrho^p}\left(\frac{\eps\boldsymbol\om}{2^j}\right)^p|Q_\varrho((A-1)\theta)|.
\end{equation*}
Next, we apply \cite[Chapter I, Lemma 2.2]{DB} with levels $k_{j+1}>k_{j}$ slicewise to $u(\cdot,t)$ for fixed $t\in( -(A-1)\theta\varrho^p,0]$.
Taking into account the measure theoretical information from \eqref{Eq:4:4}, which implies
\[
	\Big|\Big\{ u(\cdot, t)< k_j\Big\}
	\cap K_{\varrho}\Big|\ge\tfrac{1}4 \nu |K_\varrho|
	\quad\mbox{ for all $t\in(-(A-1)\theta \varrho^p,0]$,}
\]
and using Hölder's inequality, we conclude that
\begin{align*}
	(k_{j+1}-k_{j})&\big|\big\{u(\cdot, t)>k_{j+1} \big\}
	\cap K_{\varrho}\big|
	\\
	&\le
	\frac{\boldsymbol\gm \varrho^{N+1}}{\big|\big\{u(\cdot, t)<k_{j}\big\}\cap K_{\varrho}\big|}	
	\int_{\{ k_{j}<u(\cdot,t)<k_{j+1}\}\cap  K_{\varrho}}\!\!\!|Du(\cdot,t)|\,\dx\\
	&\le
	\frac{\boldsymbol \gm\varrho}{\nu}
	\bigg[\int_{\{k_{j}<u(\cdot,t)<k_{j+1}\}\cap K_{\varrho}}\!\!\!|Du(\cdot,t)|^p\,\dx\bigg]^{\frac1p}
	\big|\big\{ k_{j}<u(\cdot,t)<k_{j+1}\big\}\cap K_{\varrho}\big|^{1-\frac1p}
	\\
	&=
	\frac{\boldsymbol \gm\varrho}{\nu}
	\bigg[\int_{\{k_j<u(\cdot,t)<k_{j+1}\}\cap K_{\varrho}}\!\!\!|Du(\cdot,t)|^p\,\dx\bigg]^{\frac1p}
	\Big[ |A_j(t)|-|A_{j+1}(t)|\big]^{1-\frac1p}.
\end{align*}
Here, we abbreviated $ A_j(t):= \big\{u(\cdot,t)>k_{j}\big\}\cap K_\varrho$.
Further, we define $A_j=\{u>k_j\}\cap Q_\varrho((A-1)\theta)$.
Integrating the preceding inequality with respect to $t$ over $(-(A-1)\theta\varrho^p,0]$ and applying H\"older's inequality slicewise leads to the measure estimate
\begin{align*}
	\frac{\eps \boldsymbol\om}{2^{j+1}}\big|A_{j+1}\big|
	&\le
	\frac{\boldsymbol \gm\varrho}{\nu}\bigg[\iint_{Q_\varrho((A-1)\theta)}|D(u-k_j)_+|^p\,\dx\dt\bigg]^\frac1p
	\big[|A_j|-|A_{j+1}|\big]^{1-\frac{1}p}\\
	&\le\boldsymbol\gm \frac{\eps \boldsymbol\om}{2^j}|Q_\varrho((A-1)\theta)|^{\frac1p}
	\big[|A_j|-|A_{j+1}|\big]^{1-\frac{1}p}.
\end{align*}
Taking the power $\frac{p}{p-1}$ on both sides, we find that
\[
	\big|A_{j+1}\big|^{\frac{p}{p-1}}\le\boldsymbol \gm|Q_\varrho((A-1)\theta)|^{\frac1{p-1}}\big[|A_j|-|A_{j+1}|\big].
\]
Finally, adding the inequalities with respect to $j$ from $0$ to $j_*-1$ we obtain that
\[
	j_* \big|A_{j_*}\big|^{\frac{p}{p-1}}\le\boldsymbol\gm\big|Q_\varrho((A-1)\theta)\big|^{\frac{p}{p-1}},
\]
which is equivalent to
\[
	\big|A_{j_*}\big|\le\frac{\boldsymbol\gm}{j_*^{\frac{p-1}p}}|Q_\varrho((A-1)\theta)|.
\]
To conclude, it suffices to replace $j_*$ by $j_*-1$ in the above line and adjust $\boldsymbol\gm$.
\end{proof}
\subsubsection{A De Giorgi-type Lemma}\label{S:4:3:3}
As in the preceding section, let $\epsilon\in (0,1)$ denote the constant from Lemma~\ref{Lm:4:2} depending only on the data.
\begin{lemma}\label{Lm:5:3}
Suppose that the assumptions \eqref{Eq:4:3a} and \eqref{Eq:4:3b} hold true.
Then, there exists a constant $\nu_1\in(0,1)$ depending only on 
the data, such that if for some $j_*>1$, the measure bound
\begin{equation*}
	\Big|\big\{\boldsymbol\mu^{+}-u\le\frac{\eps\boldsymbol\om}{2^{j_*}}\Big\} 
	\cap Q_{\varrho}((A-1)\theta)
	\Big|
	\le\nu_1|Q_{\varrho}((A-1)\theta)|,
\end{equation*}
holds true, where $A=2^{j_*(p-2)}+1$ and $\theta=(\frac14\boldsymbol\om)^{q+1-p}$, then 
\[
	\boldsymbol\mu^{+}-u\ge\frac{\eps\boldsymbol\om}{2^{j_*+1}}\quad\mbox{a.e. in $Q_{\frac12\varrho}((A-1)\theta)$.}
\]
\end{lemma}
\begin{proof} 
Let $M:=2^{-j_*}\eps\boldsymbol\om$ and define
\begin{align*}
	\left\{
	\begin{array}{c}
	\displaystyle k_n=\boldsymbol\mu^+-\frac{M}2-\frac{M}{2^{n+1}},\quad \tilde{k}_n=\frac{k_n+k_{n+1}}2,\\[5pt]
	\displaystyle \varrho_n=\frac{\varrho}2+\frac{\varrho}{2^{n+1}},
	\quad\tilde{\varrho}_n=\frac{\varrho_n+\varrho_{n+1}}2,\\[5pt]
	\displaystyle K_n=K_{\varrho_n},\quad \widetilde{K}_n=K_{\tilde{\varrho}_n},\\[5pt] 
	\displaystyle Q_n=Q_{\rho_n}((A-1)\theta),\quad
	\widetilde{Q}_n=Q_{\tilde\rho_n}((A-1)\theta).
	\end{array}
	\right.
\end{align*}
We employ the energy estimate from Proposition \ref{Prop:2:1} with cutoff functions $\z$ that vanish on the parabolic boundary of $Q_{n}$,
equal identity in $\widetilde{Q}_{n}$ and fulfill
\begin{equation*}
	|D\z|\le\boldsymbol\gm\frac{2^n}{\varrho}\quad\text{ and }\quad |\z_t|\le\boldsymbol\gm\frac{2^{pn}}{(A-1)\theta\varrho^p}.
\end{equation*}
Using the condition $\frac14\boldsymbol\om\le \boldsymbol\mu^+\le 2\boldsymbol\om$ to estimate the terms on the right-hand side, we find that
\begin{align*}
	\boldsymbol\om^{q-1}\essup_{-(A-1)\theta\tilde\varrho_n^p<t<0} 
	&
	\int_{\widetilde{K}_n}\big(u-\tilde{k}_n\big)^2_+\,\dx
	+
	\iint_{\widetilde{Q}_n}\big|D\big(u-\tilde{k}_n\big)_+\big|^p\,\dx\dt\\
	&\le
	\boldsymbol\gm\frac{2^{pn}}{\varrho^p}M^{p}\left(1+\frac{\boldsymbol\om^{q-1}}{(A-1)\theta M^{p-2}}\right)|A_n|
	=
	\boldsymbol\gm\frac{2^{pn}}{\varrho^p}M^{p}\big(1+\eps^{2-p}\big)|A_n|,
\end{align*}
where we abbreviated 
\[
	A_n=\big\{u>k_n\big\}\cap Q_n.
\]
Taking into account the choice of $\z$, by an application of the Sobolev imbedding
\cite[Chapter I, Proposition~3.1]{DB} and the preceding estimate we conclude that
\begin{align*}
	&\bigg(\frac{M}{2^{n+3}}\bigg)^p
	|A_{n+1}|\le \iint_{\widetilde{Q}_n}\!\!\!\big(u-\tilde{k}_n\big)_+^p\z^p\,\dx\dt\\
	&
	\le
	\bigg[\iint_{\widetilde{Q}_n}\!\!\!\big[(u-\tilde{k}_n)_+\z\big]^{p\frac{N+2}{N}}\,\dx\dt\bigg]^{\frac{N}{N+2}}
	|A_n|^{\frac{2}{N+2}}\\
	&\le
	\boldsymbol
	\gm\bigg[\iint_{\widetilde{Q}_n}\!\!\!\big|D\big[(u-\tilde{k}_n)_+\z\big]\big|^p\,\dx\dt\bigg]^{\frac{N}{N+2}}
	\bigg[\essup_{-(A-1) \theta\tilde{\varrho}_n^p<t<0}\int_{\widetilde{K}_n}\!\!\!\big(u-\tilde{k}_n\big)^2_-\,\dx
	\bigg]^{\frac{p}{N+2}}|A_n|^{\frac{2}{N+2}}\\
	&\le 
	\boldsymbol\gm\boldsymbol\om^{\frac{p(1-q)}{N+2}}
	\bigg(\frac{2^{pn}}{\varrho^p}M^{p}\bigg)^{\frac{N+p}{N+2}}
	\big(1+\eps^{2-p}\big)^{\frac{N+p}{N+2}}
	|A_n|^{1+\frac{p}{N+2}}.
\end{align*}
Hence, for the quantity $\boldsymbol Y_n=|A_n|/|Q_n|$ we deduce the recursive inequality
\begin{align*}
	\boldsymbol Y_{n+1}
	&\le
	\boldsymbol\gm \boldsymbol b^n
	\left(\frac{(A-1)\theta M^{p-2}}{\boldsymbol\om ^{q-1}}\right)^{\frac{p}{N+2}}
	\big(1+\eps^{2-p}\big)^{\frac{N+p}{N+2}}
	\boldsymbol Y_n^{1+\frac{p}{N+2}} \\
	&=
	\boldsymbol\gm \boldsymbol b^n
	\eps^{\frac{p(p-2)}{N+2}}
	\left(1+\eps^{2-p}\right)^{\frac{N+p}{N+2}}
	\boldsymbol Y_n^{1+\frac{p}{N+2}},
\end{align*}
where $\boldsymbol b= 2^\frac{p (2N + p + 2)}{N+2}$ and $\boldsymbol \gm$ only depends on the data.
Thus, the lemma on fast geometric convergence, i.e.~\cite[Chapter I, Lemma~4.1]{DB}, ensures the existence of a constant $\nu_1\in(0,1)$
depending only on the data such that $\boldsymbol Y_n\to0$ if we assume that the smallness condition $\boldsymbol Y_o\le \nu_1$ holds true.
\end{proof}

At this stage, we conclude the reduction of oscillation in the remaining case where \eqref{Eq:4:3a} and \eqref{Eq:4:3b} hold.
To this end, denote by $\epsilon\in(0,1)$, $\boldsymbol\gamma>0$ and $\nu_1\in(0,1)$ the corresponding constants from Lemmas~\ref{Lm:4:2}, \ref{Lm:A:4} and \ref{Lm:5:3}.
Choose a positive integer $j_*$ large enough that
\begin{equation*}
	\frac{\boldsymbol\gm}{j_*^{\frac{p-1}p}}
	\le
	\nu_1 
\end{equation*}
and $Q_{\frac12\varrho}((A-1)\theta)\supset Q_{\frac14\rho}(\theta)$, where $A=2^{j_*(p-2)}+1$.
Hence, applying in turn Lemmas~\ref{Lm:A:4} and \ref{Lm:5:3}, we arrive at 
\[
	\boldsymbol\mu^{+}-u
	\ge
	\frac{\eps\boldsymbol\om}{2^{j_*+1}}\quad\mbox{a.e. in $Q_{\frac14\rho}(\theta)$.}
\]
This gives the reduction of oscillation
\[
	\essosc_{Q_{\frac14\rho}(\theta)}u
	\le
	\Big(1-\frac{\eps\boldsymbol}{2^{j_*+1}}\Big)\boldsymbol\om.
\]
Recall the reduction of oscillation achieved in Section~\ref{S:6:2} via arguments of Section~\ref{S:case-1}.
Namely, $\frac12\xi$ is chosen in the reduction of oscillation \eqref{Eq:reduc-osc-int-1} and
$\frac12\eta_o$ is chosen in the reduction of oscillation \eqref{Eq:reduc-osc-int-3}.
Combining all cases gives the reduction of oscillation exactly as in
\eqref{Eq:reduc-osc-int-4} with the choice
\[\eta=\min\Big\{\frac{\xi}2,\,\frac{\eta_o}2,\, \frac{\eps\boldsymbol}{2^{j_*+1}}\Big\},\] 
from which the rest of the proof can be reproduced just like in Section~\ref{S:5}.
\section{Proof of Boundary Regularity} 
Since Theorems~\ref{Thm:1:2} -- \ref{Thm:1:4} can be proved in a similar way as interior H\"older continuity, we will only give sketchy proofs, where we keep reference to the tools and strategies used in the interior case and highlight the main differences.
\subsection{Proof of Theorem~\ref{Thm:1:2}}\label{S:Initial}
Consider the cylinder $Q_o=K_{\rho}(x_o)\times(0,\rho^{p-1}]\subset E_T$
whose vertex $(x_o,0)$ is attached to the initial boundary $E \times \{ 0 \}$.
For ease of notation assume $x_o=0$ and
set
\begin{equation*}
	\boldsymbol \mu^+=\essup_{Q_o}u,
	\quad
	\boldsymbol\mu^-=\essinf_{Q_o}u,
	\quad
	\boldsymbol\om\ge\boldsymbol\mu^+-\boldsymbol\mu^-.
\end{equation*}
Let $\theta=(\frac14\boldsymbol\om)^{q+1-p}$.
We may assume that
\begin{equation*}
Q_\rho(\theta)\subset Q_o=K_\rho\times(-\rho^{p-1},0],\quad\text{ such that }\quad \essosc_{Q_{\rho}(\theta)}u\le\boldsymbol\om;
\end{equation*}
otherwise we would have 
\begin{equation*}
\boldsymbol\om\le 4\rho^{\frac1{p-q-1}}.  
\end{equation*}
Like in the proof of interior regularity, we start by distinguishing between the main cases
\begin{equation*}
\left\{
\begin{array}{c}
\mbox{when $u$ is near zero:  $\boldsymbol\mu^-\le\boldsymbol\om$ and 
$\boldsymbol\mu^+\ge-\boldsymbol\om$};\\[5pt]
\mbox{when $u$ is away from zero: $\boldsymbol\mu^->\boldsymbol\om$ or $\boldsymbol\mu^+<-\boldsymbol\om$.}
\end{array}\right.
\end{equation*}
The second case reduces to the corresponding estimate for weak solutions to parabolic $p$-Laplacian
equations; see
\cite[Chapter III, Lemma~11.1]{DB}.
In the first case, which implies $|\boldsymbol\mu^{\pm}|\le2\boldsymbol\om$, we proceed by a comparison to the initial datum $u_o$.
More precisely, we assume that
\[
	\mbox{either \quad$\dsty\boldsymbol \mu^+-\tfrac14\boldsymbol \om>\sup_{K_\rho}u_o\;\;$ 
	or
	$\;\;\dsty\boldsymbol \mu^-+\tfrac14\boldsymbol \om<\inf_{K_\rho}u_o$}
\]
since otherwise, we would obtain the bound
\[
\essosc_{Q_o}u\le2\essosc_{K_\rho}u_o.
\]
As both cases can be treated analogously, we consider only the second inequality with $\boldsymbol \mu^-$ and work with $u$ as a super-solution.
Using $|\boldsymbol\mu^{-}|\le2\boldsymbol\om$, Lemma~\ref{Lm:DG:initial:1} (with $\xi=\frac14$)
yields a constant $\nu_o\in(0,1)$ depending only on the data, such that
\[
	u\ge\boldsymbol\mu^-+\tfrac18\boldsymbol\om\quad\mbox{a.e.~in 
	$\widehat{Q}_1:=K_{\frac12\rho}\times\big(0,\nu_o\theta\rho^p\big]$.}
\]
Thus, we arrive at the reduction of oscillation
\[
	\essosc_{\widehat{Q}_1}u\le\tfrac78\boldsymbol\om .
\]
Finally, taking the initial datum into account, we conclude that
\[
	\essosc_{\widehat{Q}_1}u\le\max\Big\{ \tfrac78\boldsymbol\om, 2\boldsymbol\om_{u_o}(\rho)\Big\}.
\]
Now we may proceed by an iteration argument as in \cite[Section~7.1]{BDL} to conclude the proof.

\subsection{Proof of Theorem~\ref{Thm:1:3}}\label{S:Dirichlet}
Consider the cylinder  $Q_o=K_{\rho}(x_o)\times(t_o-\rho^{p-1},t_o]$
whose vertex $(x_o,t_o)$ is attached to $S_T$.
Suppose that $\rho$ is so small that $t_o-\rho^{p-1}>0$ and $\rho<\rho_o$
where $\rho_o$ is the constant from the geometric condition \eqref{geometry}.
Further, we assume that $(x_o,t_o)=(0,0)$ for ease of notation and define
\begin{equation*}
	\boldsymbol \mu^+=\essup_{Q_o\cap E_T}u,
	\quad
	\boldsymbol\mu^-=\essinf_{Q_o\cap E_T}u,
	\quad
	\boldsymbol\om\ge\boldsymbol\mu^+-\boldsymbol\mu^-.
\end{equation*}
Let $\theta=(\frac14\boldsymbol\om)^{q+1-p}$. For some $A>1$ to be determined in terms of the data, we may assume that
\begin{equation*}
Q_\rho(A\theta)\subset Q_o,\quad\text{ such that }\quad \essosc_{Q_{\rho}(A\theta)\cap E_T}u\le\boldsymbol\om;
\end{equation*}
otherwise we would have 
\begin{equation*}
\boldsymbol\om\le L\rho^{\frac1{p-q-1}}\quad\text{ where }L=4A^{\frac1{p-q-1}}.
\end{equation*}

As in the proof of interior H\"older continuity, we consider the main cases 
\begin{equation}\label{Hp-main-Dirichlet}
\left\{
\begin{array}{c}
\mbox{when $u$ is near zero:  $\boldsymbol\mu^-\le\xi\boldsymbol\om$ and 
$\boldsymbol\mu^+\ge-\xi\boldsymbol\om$};\\[5pt]
\mbox{when $u$ is away from zero: $\boldsymbol\mu^->\xi\boldsymbol\om$ or $\boldsymbol\mu^+<-\xi\boldsymbol\om$.}
\end{array}\right.
\end{equation}
Here $\xi\in(0,1)$ will be fixed in terms of the data and $\al_*$,
where $\al_*$ comes from the geometric condition \eqref{geometry} of $\pl E$.

When \eqref{Hp-main-Dirichlet}$_1$ holds true, we either arrive at the bound
\[
\essosc_{Q_o\cap E_T}u\le2\essosc_{Q_o\cap S_T}g
\]
or we continue with a comparison to the boundary datum $g$, i.e. we are concerned with the cases
\begin{equation*}
	\mbox{either}\quad\boldsymbol \mu^+-\tfrac14\boldsymbol \om>\sup_{Q_o\cap S_T}g\;
	\quad\mbox{or} \quad
	\; \boldsymbol \mu^-+\tfrac14\boldsymbol \om<\inf_{Q_o\cap S_T}g.
\end{equation*}
Since the inequalities can be treated analogously, let us consider only the second one.
Observe that $k$ satisfies \eqref{Eq:3:3}$_2$ with $Q_{R,S}$ replaced by $Q_o$,
since $(u-k)_-$ vanishes on $Q_o\cap S_T$ for all $k\le\boldsymbol \mu^-+\tfrac14\boldsymbol \om$.
Therefore, we may employ the energy estimate in Proposition~\ref{Prop:2:2} for super-solutions if we extend all integrals in the energy estimates to zero outside of $E_T$.
The extended function $(u-k)_-$, which will be denoted by the same symbol, is still contained in the functional space in \eqref{Eq:1:3p} within $Q_o$.

The proof of \cite[Lemma~4.2]{BDL}
can be adapted to the current situation, bearing in mind that
we have assumed $\pl E$ fulfills the property of positive geometric density \eqref{geometry}, and therefore
for $k=\boldsymbol \mu^-+\tfrac14\boldsymbol \om$, we have
\begin{equation}\label{Eq:7:3}
\Big|\Big \{ u^-_k(\cdot, t)-\boldsymbol\mu^-\ge\tfrac14\boldsymbol\om\Big\}\cap K_\rho(x_o)\Big|\ge\al_*|K_\rho|\quad
	\mbox{for all $t\in(-A\theta\rho^p,0]$.}
\end{equation}
Here we have used $u^-_k$ as the extension of $u$ to the whole $Q_o$ defined by
\begin{equation*}
u^-_k:=\left\{
\begin{array}{cl}
k-(u-k)_-\quad&\text{ in }Q_{o}\cap E_T,\\[5pt]
k\quad&\text{ in }Q_{o}\setminus E_T.
\end{array}\right.
\end{equation*}
By Lemma~\ref{Lm:A:2} the extension $u^-_k$ turns out to be a local, weak super-solution to \eqref{Eq:1:1}
in $Q_o$, with a properly extended principle part $\widetilde{\bf A}$, cf.~Appendix~\ref{Append:1}.
The extended $\widetilde{\bf A}$ enjoys the same type of structural conditions as in \eqref{Eq:1:2p}.
For simplicity we still use $u$ to denote the extended function in what follows.

Consequently, like in \cite[Lemma~4.2]{BDL}, there exists $\boldsymbol\gm$ depending only on the data and $\al_*$,  
 such that for any positive integer $j_*$, we have
\begin{equation*}
	\bigg|\bigg\{
	u-\boldsymbol \mu^-\le\frac{\boldsymbol \om}{2^{j_*+2}}\bigg\}\cap \widehat Q_\rho\bigg|
	\le\frac{\boldsymbol\gm}{j_*^{\frac{p-1}p}}|\widehat Q_\rho|,
\end{equation*}
where
\[
	 \widehat Q_\rho=K_{\rho}\times \big(-(2^{-j_*-2}\boldsymbol \om)^{q+1-p}\rho^p, 0\big),
\]
provided $|\boldsymbol \mu^-|\le 2^{-j_*-2}\boldsymbol \om$.
Assuming this condition on  $\boldsymbol \mu^-$ is fulfilled and letting $\nu$
be the number determined in Lemma~\ref{Lm:DG:1}, we may choose $j_*$ to satisfy that
$\boldsymbol\gm j_*^{-\frac{p-1}p}\le\nu$. Then setting 
$$\xi=2^{-j_*-2},\quad A_1=2^{j_*(p-q-1)},$$ 
Lemma~\ref{Lm:DG:1} implies that
\[
u-\boldsymbol \mu^-\ge\tfrac12\xi\boldsymbol \om\quad\text{ a.e. in } Q_{\frac12\rho}(A_1\theta),
\]
which in turn yields  that  
\[
\essosc_{\widehat Q_{\frac12\rho}\cap E_T}u\le (1-\tfrac12\xi)\boldsymbol \om.
\]
 
Hence,
the oscillation is reduced when $|\boldsymbol\mu^-|<\xi\boldsymbol\om$
for some $\xi\in(0,1)$ determined by the data and $\al_*$.
To proceed, one still needs to handle the situation when $\boldsymbol\mu^-<-\xi\boldsymbol\om$
since this is not excluded in \eqref{Hp-main-Dirichlet}$_1$.

Our current hypothesis to proceed consists of the measure information \eqref{Eq:7:3}
and $-2\boldsymbol\om<\boldsymbol\mu^-<-\xi\boldsymbol\om$
as we have assumed $\boldsymbol\mu^+\ge-\xi\boldsymbol\om$ in \eqref{Hp-main-Dirichlet}$_1$.
Departing from this, we have two ways to proceed: one is to use the expansion of positivity (Lemma~\ref{Lm:expansion:p});
the other is to follow the arguments in Sections~\ref{S:4:3:2} -- \ref{S:4:3:3}.
We only describe the first option. 

In fact, by Lemma~\ref{Lm:expansion:p}, the measure information \eqref{Eq:7:3} translates into the pointwise estimate
\[
u\ge\boldsymbol \mu^-+\eta\boldsymbol \om\quad\text{ a.e. in }Q_{\frac12\rho}(A_2\theta)
\]
for some $\eta\in(0,1)$ depending on the data and $\xi$.
This gives us a reduction of oscillation as usual, and hence finishes
the reduction of oscillation under the condition \eqref{Hp-main-Dirichlet}$_1$
The constant $A_2$ is determined by the data in this step, through $b$ and $\eta$ of Lemma~\ref{Lm:expansion:p}.
The final choice of $A$ is given by the larger one of $A_1$ and $ A_2$.

As in the interior case, we repeat the arguments inductively until the second case of \eqref{Hp-main-Dirichlet} is satisfied for some index $j$ for the first time.
Starting from $j$, the equation behaves like the parabolic $p$-Laplacian type equation within $Q_j\cap E_T$.
In order to render this point technically, we adapt the proof for interior regularity, where we use in particular the boundary regularity result \cite[Proposition~7.2]{BDL} for the parabolic $p$-Laplacian near the lateral boundary.

\subsection{Proof of Theorem~\ref{Thm:1:4}}\label{S:Neumann}
First of all, we observe that the second proof of interior regularity (Theorem~\ref{Thm:1:1}) in Section~\ref{S:6} is based solely on
the energy estimates in Proposition~\ref{Prop:2:1} and a corresponding
H\"older estimate for solutions to the parabolic $p$-Laplacian.

A key ingredient -- the Sobolev imbedding (cf. \cite[Chapter I, Proposition~3.1]{DB}) -- was used in order to establish
Lemma~\ref{Lm:DG:1}, Lemma~\ref{Lm:DG:initial:1}, Lemma~\ref{Lm:DG:initial:2}, Lemma~\ref{Lm:4:2} and Lemma~\ref{Lm:5:3}, assuming the functions $(u-k)_{\pm}\z^p$ vanish on the lateral boundary of the domain of integration.
This assumption in turn is fulfilled  by choosing a proper cutoff function $\z$.
In the boundary situation similar arguments have been employed in Section~\ref{S:Dirichlet} or in Section~\ref{S:Initial}
by restricting the value of the level $k$ according to the Dirichlet data as in \eqref{Eq:3:3}
or the initial data as in \eqref{Eq:3:2}.

However, in the current situation of Neumann data the functions $(u-k)_{\pm}\z^p$ under conditions of Proposition~\ref{Prop:2:4} do not vanish on $S_T$ and therefore such a Sobolev imbedding cannot be used in general.
On the other hand, a similar Sobolev imbedding (cf. \cite[Chapter. I, Proposition 3.2]{DB})
 that does not require functions to vanish on the boundary
still holds for the functional space
\begin{equation*}  
	u\in C\big(0,T;L^p(E)\big)\cap L^p\big(0,T; W^{1,p}(E)\big).
\end{equation*}
The appearing constant now depends on $N$, 
the structure of $\pl E$ and the ratio $T/|E|^{\frac{p}N}$,
which is invariant for cylinders of the type $Q_\rho=K_{\rho}\times(-\rho^p,0]$
and $Q_\rho\cap E_T$ as well, provided $\pl E$ is smooth enough.
In particular, Lemma~\ref{Lm:DG:1}, Lemma~\ref{Lm:DG:initial:1}, Lemma~\ref{Lm:DG:initial:2}, Lemma~\ref{Lm:4:2}
and Lemma~\ref{Lm:5:3}
can be proved in this boundary setting.

Finally, we remark that the use of De Giorgi's isoperimetric inequality (cf. \cite[Chapter I, Lemma~2.2]{DB} and \cite[Theorem 4.2.1]{Ziemer})
is permitted for extension domains, and thus in particular for $C^1$-domains. Thus the machinery used in Lemma~\ref{Lm:A:4} can be reproduced.

For the proof of Theorem~\ref{Thm:1:4} we now consider a cylinder  $Q_o=K_{\rho}(x_o)\times(t_o-\rho^{p-1},t_o]$ whose vertex $(x_o,t_o)$ is attached to $S_T$ and $\rho$ is so small that $t_o-\rho^{p-1}>0$.
According to the preceding considerations we proceed exactly as in the second proof of interior regularity in Section~\ref{S:6}. Obviously, in the present situation all cylinders have to be intersected with $E_T$. In this way, we conclude a reduction of oscillation for the lateral boundary point $(x_o,t_o)$.

\appendix

\section{On the Notion of Parabolicity}\label{Append:1}
We collect some useful lemmas regarding the notion of parabolicity for
\eqref{Eq:1:1} -- \eqref{Eq:1:2p}.
\begin{lemma}\label{Lm:A:1}
Let $u$ be a local weak sub(super)-solution to \eqref{Eq:1:1} -- \eqref{Eq:1:2p}.
Then, for any $k\in\rr$, the truncation $k\pm(u-k)_\pm$
is a local weak sub(super)-solution to \eqref{Eq:1:1} -- \eqref{Eq:1:2p}.
\end{lemma}

The analysis has been carried out in \cite[Appendix~A]{BDL} for $q=p-1$.
However, the same proof actually works for all $p>1$ and $q>0$ after minor changes.

In particular, when $u$ is a local weak solution,
$u_{+}$ and $u_-$ are non-negative, local weak sub-solutions to \eqref{Eq:1:1} -- \eqref{Eq:1:2p}.
By \cite[Theorem~4.1]{BHSS}, they are locally bounded and hence $u$ is also.

In order to formulate an analog of Lemma~\ref{Lm:A:1} near the lateral boundary $S_T$ for a sub(super)-solution $u$ to \eqref{Dirichlet},
consider the cylinder $Q_{R,S}=K_R(x_o)\times(t_o-S,t_o)$ whose vertex $(x_o,t_o)$ is attached to $S_T$.
Further, for a level $k$ satisfying \eqref{Eq:3:3}, we are concerned with
the following truncated extension of $u$ in $Q_{R,S}$:
\begin{equation*}
u_k^{\pm}:=\left\{
\begin{array}{cl}
k\pm(u-k)_\pm\quad&\text{ in }Q_{R,S}\cap E_T,\\[5pt]
k\quad&\text{ in }Q_{R,S}\setminus E_T.
\end{array}\right.
\end{equation*}
Moreover, the extension of $\bl{A}$ defined by
\begin{equation*}
\widetilde{\mathbf A}(x,t,u,\z):=\left\{
\begin{array}{cl}
\mathbf A(x,t,u,\z)\quad&\text{ in }Q_{R,S}\cap E_T,\\[5pt]
|\z|^{p-2}\z\quad&\text{ in }Q_{R,S}\setminus E_T
\end{array}\right.
\end{equation*}
is a Carath\'eodory function satisfying \eqref{Eq:1:2p} with structure constants $C_o$
and $C_1$ replaced by $\min\{1,C_o\}$ and $\max\{1,C_1\}$, respectively.
In this situation, the following lemma holds.
\begin{lemma}\label{Lm:A:2}
Suppose $u$ is a sub(super)-solution to \eqref{Dirichlet} with \eqref{Eq:1:2p}
and the level $k$ satisfies \eqref{Eq:3:3}.
Let $u_k^{\pm}$ be defined as above.
Then $u^{\pm}_k$ is a local weak sub(super)-solution to \eqref{Eq:1:1} with $\widetilde{\mathbf A}$
in $Q_{R,S}$.
\end{lemma}


\end{document}